\documentclass{amsart}
\usepackage{latexsym,amsxtra,amscd,ifthen}
\usepackage{amsfonts}
\usepackage{verbatim}
\usepackage{amsmath}
\usepackage{amsthm}
\usepackage{amssymb}

\numberwithin{equation}{section}

\theoremstyle{plain}
\newtheorem{theorem}{Theorem}[section]
\newtheorem{lemma}[theorem]{Lemma}
\newtheorem{proposition}[theorem]{Proposition}
\newtheorem{corollary}[theorem]{Corollary}

\theoremstyle{definition}
\newtheorem{definition}[theorem]{Definition}
\newtheorem{example}[theorem]{Example}

\newtheorem{remark}[theorem]{Remark}
\newtheorem{question}[theorem]{Question}

\newcommand{\ds}{\displaystyle}
\makeatletter              
\let\c@equation\c@theorem  
\makeatother

\DeclareMathOperator{\hdet}{hdet}

\DeclareMathOperator{\gldim}{gldim}

\DeclareMathOperator{\uExt}{\underline{Ext}}

\DeclareMathOperator{\Ext}{Ext}

\DeclareMathOperator{\gr}{gr}

\DeclareMathOperator{\Proj}{Proj}

\DeclareMathOperator{\Aut}{Aut_{gr}}

\DeclareMathOperator{\GKdim}{GKdim}

\begin{document}

\title{Rigidity of graded regular algebras}

\author{E. Kirkman, J. Kuzmanovich and J.J. Zhang}

\address{Kirkman: Department of Mathematics,
P. O. Box 7388, Wake Forest University,
Winston-Salem, NC 27109}

\email{kirkman@wfu.edu}

\address{Kuzmanovich: Department of Mathematics,
P. O. Box 7388, Wake Forest University,
Winston-Salem, NC 27109}

\email{kuz@wfu.edu}

\address{zhang: Department of Mathematics, Box 354350,
University of Washington, Seattle, Washington 98195,
USA}

\email{zhang@math.washington.edu}

\begin{abstract}
We prove a graded version of Alev-Polo's rigidity theorem:
the homogenization of the universal enveloping algebra
of a semisimple Lie algebra and the Rees ring of
the Weyl algebras $A_n(k)$ cannot be isomorphic to their
fixed subring under any finite group action. We also show
the same result for other classes of graded regular
algebras including the Sklyanin algebras.
\end{abstract}








\maketitle


\setcounter{section}{-1}
\section{Introduction}
\label{sec0}

The invariant theory of $k[x_1, \cdots, x_n]$ is a rich
subject whose study has motivated many developments in
commutative algebra and algebraic geometry. One important
result is the Shephard-Todd-Chevalley Theorem [Theorem
\ref{xxthm1.1}] that gives necessary and sufficient
conditions for the fixed subring $k[x_1, \cdots, x_n]^G$
under a finite subgroup $G$ of $GL_n(k)$ to be a polynomial
ring. The study of the invariant theory of noncommutative
algebras is not understood well, and it is reasonable to
begin with the study of finite groups acting on rings that
are seen as generalizations of polynomial rings.

We will show that in contrast to the commutative case,
a noncommutative regular algebra  $A$ is often {\it rigid},
meaning that $A$ is not isomorphic to any fixed subring
$A^G$ under a non-trivial group of automorphisms $G$ of
$A$. A typical result is the Alev-Polo rigidity theorem
that shows that both the universal enveloping algebra of
a semisimple Lie algebra and the Weyl algebras $A_n(k)$
are rigid algebras.

\begin{theorem}[Alev-Polo rigidity theorem \cite{AP}]
\label{xxthm0.1}
$\quad$
\begin{enumerate}
\item
Let ${\mathfrak g}$ and ${\mathfrak g}'$ be two
semisimple Lie algebras. Let $G$ be a finite group of
algebra automorphisms of $U({\mathfrak g})$ such that
$U({\mathfrak g})^G \cong U({\mathfrak g}')$. Then $G$
is trivial and ${\mathfrak g}\cong {\mathfrak g}'$.
\item
If $G$ is a finite group of algebra
automorphisms of $A_n(k)$ then the fixed subring
$A_n(k)^G$ is isomorphic to $A_n(k)$ only when $G$
is trivial.
\end{enumerate}
\end{theorem}

The main goal of this paper is to investigate a similar
question for graded algebras. As one example, in Section
\ref{sec6} we prove the following graded version of the
Alev-Polo rigidity theorem. Let $H({\mathfrak g})$ denote
the homogenization of the universal enveloping algebra of
a finite dimensional Lie algebra ${\mathfrak g}$ (the
definition is given in Section \ref{sec6}).

\begin{theorem}
\label{xxthm0.2}
\begin{enumerate}
\item
Let ${\mathfrak g}$ and ${\mathfrak g}'$ be Lie
algebras with no 1-dimensional Lie ideal. Let $G$ be
a finite group of graded algebra automorphisms of
$H({\mathfrak g})$ such that $H({\mathfrak g})^G\cong
H({\mathfrak g}')$ (as ungraded algebras). Then $G$
is trivial and ${\mathfrak g}\cong {\mathfrak g}'$.
\item
Let $A$ be the Rees ring of the Weyl algebra
$A_n(k)$ (with respect to the standard filtration of
$A_n(k)$). Then $A$ is not isomorphic to $A^G$
(as ungraded algebras) for any finite non-trivial group
of graded automorphisms.
\end{enumerate}
\end{theorem}

Artin-Schelter regular algebras [Definition \ref{xxdefn1.5}]
are a class of graded algebras that are generalizations of
polynomial algebras, and they have been used in many
areas of mathematics and physics. One can ask whether
an Artin-Schelter regular algebra $A$ can be isomorphic to a
fixed subring $A^G$ when $G$ is a non-trivial finite
group of graded algebra automorphisms of $A$.  One
could consider fixed rings under ungraded automorphisms
(\cite{AP} did not restrict itself to filtered
automorphisms) also, but we leave that problem to others.
Although it is easy to construct noncommutative algebras
$A$ and groups of automorphisms $G$ where $A^G$ is
isomorphic to $A$ [Example \ref{xxex1.2}], it turns
out that this happens less often than we expected
[Lemma \ref{xxlem5.2}(b)], and we will provide both
some necessary conditions and some sufficient conditions
for this problem.  Our work thus far suggests that a
generalization of the Shephard-Todd-Chevalley Theorem
requires a new notion of reflection group, one that
depends on the Hilbert series of the Artin-Schelter
regular algebra $A$
(for the conditions used in the commutative case turn
out to be neither necessary nor sufficient
[Example \ref{xxex2.3}]). In this paper we focus on
Artin-Schelter regular algebras that have the same
Hilbert series as commutative polynomial rings.
We call $A$ a {\it quantum polynomial ring}
({\it of dimension $n$}) if it is a noetherian, graded,
Artin-Schelter regular domain of global dimension $n$,
with Hilbert series $(1-t)^{-n}$. Skew polynomial rings,
$H({\mathfrak g})$, the Rees rings of the Weyl algebras,
and Sklyanin algebras are all quantum polynomial rings.
One of our results is the following.

\begin{theorem}[Theorem \ref{xxthm6.2}]
\label{xxthm0.3}
Let $A$ be a quantum polynomial ring. Suppose that
there is no nonzero element $b\in A_1$ such that
$b^2$ is normal in $A$. Then $A$ is not isomorphic
to $A^G$ as ungraded algebras for any non-trivial
finite group $G$ of graded algebra automorphisms.
\end{theorem}

If $A$ is viewed as the coordinate ring of a noncommutative
affine $n$-space, then Theorem \ref{xxthm0.3} can be
interpreted as: a ``very noncommutative'' affine $n$-space
cannot be isomorphic to any quotient space of itself
under a non-trivial finite group action. If we really
understood noncommutative spaces, this might be a
simple fact. The hypothesis that $A$ has no
normal element of the form $b^2$ is easy to check in many
cases. For example, Theorem \ref{xxthm0.3} applies to the
non-PI Sklyanin algebras of dimension $n$.

\begin{corollary}[Corollary \ref{xxcor6.3}]
\label{xxcor0.4}
Let $S$ be a non-PI Sklyanin algebra of global dimension
$n \geq 3$. Then $S$ is not isomorphic to $S^G$ for
any non-trivial finite group $G$ of graded algebra
automorphisms.
\end{corollary}

The method of proving Theorems \ref{xxthm0.2}(a) and
\ref{xxthm0.3} is to show that $H({\mathfrak g})^G$ and $A^G$ do
not have finite global dimension for any non-trivial $G$. This
method applies to other algebras such as down-up algebras (see
Proposition \ref{xxprop6.4}) which are not quantum polynomial
rings.  However, if $A$ is the Rees ring of the Weyl algebra
$A_n(k)$ then there are groups $G$ of automorphisms of $A$ so that
$A$ has a fixed subring $A^G$ that is Artin-Schelter regular, but
not isomorphic to $A$ [Example \ref{xxex5.4}].  Since commutative
polynomial rings are the only commutative (Artin-Schelter) regular
algebras, the situation where $A^G$ is Artin-Schelter regular, but
not isomorphic to $A$, does not arise in the commutative case.
Hence this paper deals with a small portion of a more fundamental
question: find all noetherian graded Artin-Schelter regular
algebras $A$ and finite groups $G$ of graded algebra automorphisms
of $A$ such that $A^G$ has finite global dimension.  Given a
well-studied quantum polynomial ring, it should be possible to
find all finite groups $G$ such that $A^G$ has finite global
dimension. Following the commutative case, we call such a group a
{\it reflection group}. For algebras in Theorems \ref{xxthm0.2}(a)
and \ref{xxthm0.3} and Corollary \ref{xxcor0.4}, there is no
non-trivial reflection group.

For the simplest noncommutative ring $k_q[x,y]$ with relation
$xy=qyx$ for a nonzero scalar $q$ in the base field $k$, all
reflection groups for $k_q[x,y]$ have been worked out completely,
and these results motivated our approach to general Artin-Schelter
regular algebras. However, the project becomes much harder when
the global dimension of the algebra $A$ is higher, and less is
known about large dimension Artin-Schelter regular algebras.

Some ideas in the classical Shephard-Todd-Chevalley
theorem for the commutative polynomial ring can be
extended to the noncommutative case. Let $A$ be
a quantum polynomial ring, and let $g$ be a graded algebra
automorphism of $A$. Then $g$ is called a
{\it quasi-reflection} of a quantum polynomial ring
of dimension $n$ if its trace is of the form
$$Tr_A(g,t)={\frac{1}{(1-t)^{n-1}(1-\xi t)}}$$
for some scalar $\xi \neq 1$.
We classify all possible quasi-reflections of quantum
polynomial rings in Theorem \ref{xxthm3.1}, which states
that, with only one interesting exception, the
quasi-reflections of a quantum polynomial ring are
reflections of the generating space $A_1$ of $A$. The notion
of quasi-reflection is extended to Artin-Schelter regular algebras,
and we prove that for any Artin-Schelter regular algebra $A$,
if $A^G$ has finite global dimension, then $G$ must contain
a quasi-reflection [Theorem \ref{xxthm2.4}]. Therefore
Theorems \ref{xxthm0.2}(a) and \ref{xxthm0.3} follow by
verifying that $H({\mathfrak g})$ in Theorem
\ref{xxthm0.2}(a) and $A$ in Theorem \ref{xxthm0.3}
do not have any quasi-reflections.  More work is
required in analyzing the fixed ring of the
Rees algebras of the Weyl algebras, as they have
quasi-reflections and Artin-Schelter regular fixed
rings [Proposition \ref{xxprop6.7} and Corollary \ref{xxcor6.8}].

As a secondary result we formulate a partial
version of Shephard-Todd-Chevalley theorem for
noncommutative Artin-Schelter regular algebras.

\begin{theorem}[Theorem \ref{xxthm5.3}]
\label{xxthm0.5}
Let $A$ be a quantum polynomial ring and let $g$ be a
graded algebra automorphism of $A$ of order $p^m$ for
some prime $p$ and some natural number $m$. Then
$A^g$ has finite global dimension if and only if
$g$ is a quasi-reflection.
\end{theorem}

We conjecture that a full version of
Shephard-Todd-Chevalley theorem for noncommutative
Artin-Schelter regular algebras holds. Some further study
about reflection groups and a noncommutative version of
Shephard-Todd-Chevalley theorem will be reported in
\cite{KKZ1}.

\section{General preparations}
\label{sec1}

In this section we review some background and
collect some definitions that we will use in later
sections.

Throughout let $k$ be a commutative base field of
characteristic zero. We assume that $k$ is algebraically
closed for the convenience of our computation,
but this assumption is not necessary for most of
the results. All vector spaces, algebras and rings
are over $k$. The opposite ring of an algebra $A$ is
denoted by $A^{op}$.

Let $V$ be a finite dimensional vector space over $k$ and let $g$
be a linear transformation of $V$. We call $g$ a {\it reflection}
of $V$ if $\dim V^g\geq  \dim V-1$, where $V^g$ is the
$g$-invariant subspace of $V$. Such a $g$ is also called a {\it
pseudo-reflection} by many authors \cite[p.~24]{Be}. We have
dropped the prefix ``pseudo'' because we will introduce several
different kinds of reflections in this paper. Let $k[V]$ denote
the symmetric algebra on $V$ -- the polynomial ring in $n$
commuting variables where $n=\dim V$. The famous
Shephard-Todd-Chevalley theorem gives necessary and sufficient
conditions for the fixed ring of a polynomial ring to be a
polynomial ring (see \cite[Theorem 7.2.1]{Be}).

\begin{theorem}[Shephard-Todd-Chevalley theorem]
\label{xxthm1.1}
Suppose $G$ is a finite group acting faithfully on
a finite dimensional vector space $V$. Then the fixed
subring $k[V]^G$ is isomorphic to $k[V]$
if and only if $G$ is generated by reflections of $V$.
\end{theorem}

A finite group $G\subset GL(V)$ is called {\it a
reflection group} of $V$ if $G$ is generated by
reflections. When the base field is ${\mathbb R}$, a
reflection group is also called a {\it Coxeter group}.
Classifications of reflection groups over different
fields are given in \cite{Co,ShT,CE}.

There are noncommutative algebras that are not rigid, i.e. have
fixed subrings isomorphic to themselves. In fact, one can
construct an algebra $A$ and a group $G$ of automorphisms of $A$
so that $A^G$ is isomorphic to $A$ using any ring $R$ and any
graded automorphism $\sigma$ of $R$ with finite order using a skew
polynomial extension in the following way.

\begin{example}
\label{xxex1.2}
If $R$ is an algebra with an automorphism $\sigma$ of
order $n$, so that $\sigma^{n+1} = \sigma$, then if
we let $\xi$ be an $(n+1)$-st root of unity and extend
$\sigma$ to the skew polynomial extension
$A=R[z;\sigma]$ by $g|_R= Id_R$ and $g(z) = \xi z$,
then the fixed subring $A^G= R[z^{n+1};\sigma^{n+1}]
\cong A$ for $G = <g>$. We note that if $R$ is
Artin-Schelter regular (defined below) then so is $A$.
\end{example}

On the other hand, Alev-Polo's result [Theorem
\ref{xxthm0.1}] and results in \cite{Sm1,AP,Jo} suggest
that it is rare that a noncommutative ring $A$ is
isomorphic to a fixed subring $A^G$ for a finite group $G$.
The motivation for this paper is the following question.

\begin{question}
\label{xxque1.3}
Under what conditions on the algebra $A$ and the group
$G$, is $A$ isomorphic to $A^G$?
\end{question}

Our focus is on graded algebras and graded automorphisms
since some combinatorial structures of graded rings and
their fixed subrings can be used to study this problem.
It follows from the Shephard-Todd-Chevalley theorem that
the commutative graded polynomial ring $k[V]$ can be
isomorphic to its fixed subrings. Hence it is expected
that some version of Shephard-Todd-Chevalley theorem will
hold for ``somewhat commutative'' polynomial rings. We
will present some examples that illustrate this idea
[Examples \ref{xxex4.4}, \ref{xxex5.4} and \ref{xxex6.6}].

In the rest of this section we review some properties of
the Hilbert series of an algebra, the trace of an
automorphism, and Artin-Schelter regular algebras, as well
as some techniques from invariant theory that will be used
in this paper.

Throughout let $A$ be a connected graded algebra, namely,
$$A=k\oplus A_1\oplus A_2 \oplus \cdots$$
where each $A_i$ is finite dimensional and $A_iA_j
\subseteq A_{i+j}$ for all $i,j$. The Hilbert series of $A$
is defined to be the formal power series
$$H_A(t)\; =\; \sum_{i\geq 0} \dim A_i\;
t^i\in {\mathbb Z}[[t]].$$
The Hilbert series of a graded $A$-module is defined similarly.
Let $\Aut(A)$ be the group of graded algebra automorphisms
of $A$. For every $g\in \Aut(A)$, the trace of $g$ \cite{JiZ}
is defined to be
$$Tr_A(g,t)= \sum_{i\geq 0} tr (g|_{A_i}) \; t^i\in k[[t]].$$
It is obvious that $Tr_A(Id_A,t)=H_A(t)$, and
the converse is clearly true for $g$ of finite order when
$\operatorname{char} k=0$.  In the next section we will
define our generalization of the notion of a ``reflection" in
terms of the trace of the automorphism.

If $g$ has finite order then $Tr_A(g,t)$ is in
${\mathbb Q}(\zeta_n)[[t]]$ where ${\mathbb Q}(\zeta_n)$ is
the cyclotomic field generated by primitive $n$-th root of
unity, $\zeta_n=e^{2\pi i/n}$. For each integer $p$ such
that $(p,n)=1$, there is an automorphism of ${\mathbb Q}
(\zeta_n)/{\mathbb Q}$ determined by
$$\Xi_p: \zeta_n\to \zeta_n^p,$$
and the Galois group $G({\mathbb Q}(\zeta_n)/{\mathbb Q})$
is generated by the $\Xi_p$. One can easily extend $\Xi_p$
to an algebra automorphism of ${\mathbb Q}(\zeta_n) [[t]]
/{\mathbb Q}[[t]]$ by applying $\Xi_p$ to the coefficients.

\begin{lemma}
\label{xxlem1.4}
Let $g$ be a graded automorphism of $A$ of order $n$.
Then, for every $p$ coprime to $n$, $Tr_A(g^p,t)=\Xi_p(Tr_A(g,t))$.
In particular, if $g\in \Aut(A)$ is of finite order, then
$Tr_A(g^{-1},t)=\overline{Tr_A(g,t)}$, where for computational
purposes we assume $k\subset {\mathbb C}$ and $\overline{f}$ is
the series whose coefficients are complex conjugates of the
coefficients of $f$.
\end{lemma}

\begin{proof} We only need to show that
$tr (g^p|_{A_i})=\Xi_p(tr (g|_{A_i}))$ for all $i$.
Since $g$ has order $n$, it is diagonalizable. Let
$\{b_1,\cdots,b_q\}$ be a basis of $A_i$ such that
$$g(b_t)=\zeta_n^{w_t} b_t$$
for some integer $w_t$, for all $t=1,\cdots,q$.
For every $p$ coprime to $n$, $\Xi_p$ is an
automorphism and
$$g^p(b_t)=\zeta_n^{pw_t}b_t=\Xi_p(\zeta_n^{w_t}) b_t.$$
Hence $tr (g^p|_{A_i})=\Xi_p(tr (g|_{A_i}))$.  The second
part follows since for a root of unity $\zeta^{-1} =
\overline{\zeta}$.
\end{proof}

We will use this lemma when $Tr_A(g,t)$ is a rational
function, viewed as an infinite power series.

The Gelfand-Kirillov dimension of an algebra $A$ is
denoted by $\GKdim A$;  it is related to the rate of growth
of the graded pieces $A_n$ of $A$ (see \cite{KL}).
The commutative polynomial ring $k[x_1, \cdots, x_n]$
has $\GKdim =n$. The Gelfand-Kirillov dimension of an
$A$-module is defined similarly. Let $\uExt_A(M,N)$ be the
usual $\Ext$-group of graded $A$-modules $M$ and $N$
with ${\mathbb Z}$-grading as defined in \cite[p.240]{AZ}.

\begin{definition}
\label{xxdefn1.5}
A connected graded algebra $A$
is called {\it Artin-Schelter Gorenstein}
if the following conditions hold:
\begin{enumerate}
\item
$A$ has graded injective dimension $d<\infty$ on
the left and on the right,
\item
$\uExt^i_A(k,A)=\uExt^i_{A^{op}}(k,A)=0$ for all
$i\neq d$, and
\item
$\uExt^d_A(k,A)\cong \uExt^d_{A^{op}}(k,A)\cong k(l)$
for some
$l$.
\end{enumerate}
If in addition,
\begin{enumerate}
\item[(d)]
$A$ has finite (graded) global dimension, and
\item[(e)]
$A$ has finite Gelfand-Kirillov dimension,
\end{enumerate}
then $A$ is called {\it Artin-Schelter regular}
(or {\it regular} for short) of dimension $d$.
\end{definition}

Note that polynomial rings $k[x_1,x_2,\cdots,x_n]$
for $n\geq 0$, with $\deg x_i>0$, are Artin-Schelter
regular of dimension $n$, and these are the only
commutative Artin-Schelter regular algebras, so
Artin-Schelter regular algebras are natural
generalizations of commutative polynomial rings.

For (Artin-Schelter) regular algebras we can say more about
the trace of an automorphism.

\begin{lemma}
\label{xxlem1.6}
Let $A$ be regular and let $g\in \Aut(A)$.
\begin{enumerate}
\item
\cite[Theorem 2.3(4)]{JiZ} $Tr_A(g,t)$ is equal to
$1/e_g(t)$, where $e_g(t)$ is a polynomial in $k[t]$
with $e_g(0)=1$. We call $e_g(t)$ the {\it Euler polynomial}
of $g$.
\item
\cite[Proposition 3.1(3)]{StZ}
$H_A(t)=1/e(t)$ where $e(t)$ is an integral polynomial.
The polynomial $e(t)$ is called the {\it Euler polynomial}
of $A$. Furthermore $e(t)$ is a product of cyclotomic polynomials.
\item \cite[Corollary 2.2]{StZ}  The multiplicity of $t=1$
as a root of the Euler polynomial of $A$ is the $\GKdim A$.
\item
\cite[Theorem 3.1]{JiZ}
The polynomials $e(t)$ and $e_g(t)$ have the same degree.
\item
Suppose $g$ has finite order and $Tr_A(g,t)=e_g(t)^{-1}$.
Then the zeroes of the polynomial $e_g(t)$ are all
roots of unity.
\end{enumerate}
\end{lemma}

\begin{proof} Only the second assertion in (b) and (e) are new.

(b) By \cite[Corollary 2.2]{StZ}, all the zeroes of the
polynomial $e(t)$ appearing in Lemma \ref{xxlem1.6}(b) are
roots of unity. Since $e(t) \in \mathbb{Z}[t]$, therefore $e(t)$ is a
product of cyclotomic polynomials.

(e) Let $n$ be the order of $g$. Let $p$ be any
integer $0<p<n$ coprime to $n$. By Lemma \ref{xxlem1.4},
$Tr_A(g^p,t)=\Xi_p(Tr_A(g,t))$. Let $e_p(t)=(Tr_A(g^p,t))^{-1}$
for all $p$. By \cite[Proposition 3.3]{JiZ},
every zero of $e_p(t)$ has absolute value $1$.
Now let
$$f(t)=\prod_{(p,n)=1} e_p(t)=\prod_{(p,n)=1}
\Xi_p(Tr_A(g,t))^{-1}$$
where the notation $(p,n)=1$ means the set of integers
$p$ such that $0<p<n$ and that $p$ is coprime to $n$.
Since all coefficients of $\Xi_p(Tr_A(g,t))^{-1}$ are
in ${\mathbb Z}[\zeta_n]$, $f(t)\in
{\mathbb Z}(\zeta_n)[t]$. By the definition of
$f(t)$, $\Xi_p(f(t))=f(t)$. Since the coefficients of
$f(t)$ are fixed by all elements of the Galois group
$G({\mathbb Q}(\zeta_n)/{\mathbb Q})$ therefore
$f(t)\in~{\mathbb Q}[t]~\cap~{\mathbb Z}[\zeta_n][t]=
{\mathbb Z}[t]$. Since every zero of $f(t)$
is an algebraic integer with it and all its conjugates of
absolute value $1$, it follows from
\cite[Corollary 2.38, p.90]{Mo} that every
zero of $e(t)$ (and hence of $e_p(t)$) is a root of unity.
\end{proof}

Next we consider the multiplicity of $t=1$ as a root
of the Euler polynomial of a finite graded automorphism
$g$ of a regular domain $A$.  We show that this
multiplicity is bounded by the $\GKdim A$, and can be
equal to $\GKdim A$ only when $g$ is the identity
automorphism.

\begin{lemma}
\label{xxlem1.7}
Let $A$ be a connected graded finitely generated
algebra, and let $M$ be a graded finitely generated right
$A$-module of $\GKdim M = n$.  Let $g$ be a graded vector space
automorphism of $M$ that has finite order and $Tr_M(g,t) =
p(t)/q(t)$, where the roots of $q(t)$ are roots of unity. Then the
multiplicity of $1$ as a root of $q(t)$ is $\leq n$.
\end{lemma}

\begin{proof}
Assume to the contrary that the multiplicity of $1$ as a
root of $q(t)$ is $\geq n+1$. Let $H_M(t) = \sum h_i t^i$
be the Hilbert series of $M$, and let $Tr_M(g,t) = \sum m_i t^i$
be the trace function of $g$ on $M$.  We note that $|m_i|
\leq h_i$ for all $i$ since $g$ has finite order so that the
eigenvalues of $g$ are roots of unity.  As in the proof of
\cite[Proposition 2.21]{ATV2}, let $p$ be the highest order
of any pole of $Tr_M(g,t)$, and express all roots of $q(t)$ as
powers of a primitive $N$th root of unity $\zeta$. We have
$$Tr_M(g,t) = \sum_{s,j}\frac{c_{s,j}}{(1-\zeta^st)^j} + f(t)$$
where $s = 0 ,\ldots, N, j = 1 ,\ldots, p$, and
for sufficiently large $i$ we have
$$\begin{aligned}
m_i &= \sum_{s,j} c_{s,j} \left(
\begin{array}{c}
i+j-1 \\j-1
\end{array}
\right) \zeta^{si}\\
&= \left( \sum_{s} c_{s,p} \zeta^{si} \right)
\frac{i^{p-1}}{(p-1)!} + \text{ terms of lower degree in $i$},
\end{aligned}
$$
and the coefficients cycle through $N$ polynomials.  Not all
the $c_{s,p}$ are zero, so the leading coefficients
$\sum c_{s,p} \zeta^{si}$ are not all zero.  Hence there is
a subsequence $m_{i_0+Ni}$ with $|m_{i_0+Ni}| \geq K(Ni)^{p-1}$
for all $i\geq 1$, for some constant $K$. Since
$h_i\geq |m_i|$ for all $i$,
$$\sum_{j\leq i_0+Ni} h_i\geq \sum_{s\leq i} |m_{i_0+Ns}|
\geq KN^{p-1} \sum_{s\leq i} s^{p-1}\geq K' i^{p}
\geq K'' (i_0+Ni)^p$$
for some $K', K''>0$. Since $A$ is a finitely generated
algebra and $M$ is a finitely generated $A$-module,
$\GKdim M\geq p$. This contradicts to the fact that
$\GKdim M=n$ and $p\geq n+1$.
\end{proof}

\begin{proposition}
\label{xxprop1.8}
Let $A$ be a regular domain. If $g\in \Aut(A)$ has
finite order, and if its Euler polynomial has $t=1$ as a
root of multiplicity equal to the $\GKdim A$, then $g$ is
the identity.
\end{proposition}

\begin{proof}
Suppose that the Euler polynomial of $g$ has $t=1$ a root
of multiplicity equal to $\GKdim A$, but that $g$ is not the
identity.  Then $g$ has an eigenvalue $\lambda \neq 1$
and an element $x \in A$ with $g(x) = \lambda x$.  Let $M =
A/xA$ and let $\bar{g}$ be the induced graded vector space
automorphism of $M$.  Then
$$Tr_{M}(\bar{g},t) = (1 - \lambda t) Tr_A(g,t),$$
and the order of the pole of $Tr_M(\bar{g},t)$ at $t=1$ is
equal to the order of the pole of $Tr_A(g,t)$ at $t=1$,
which is by assumption the $\GKdim A$.  But $\GKdim A >
\GKdim M$ by \cite[Proposition 8.3.5]{MR}, contradicting
Lemma \ref{xxlem1.7}.
\end{proof}

Associated to a graded automorphism $g$ of an Artin-Schelter
Gorenstein algebra $A$ is a constant $\hdet_A g$ defined by
J{\o}rgensen-Zhang \cite{JoZ}, and the map
$\hbox{hdet}_A:{\Aut}(A) \rightarrow k^{\times}$
 defines a group homomorphism.  It follows from
\cite[Lemma 2.6 and Theorem 4.2]{JoZ} that when $A$ is a
regular algebra then the $\hbox{hdet } g$ can be computed
from the trace of $g$: since $Tr_A(g,t)$ is a rational
function in $t$ it can be written as a Laurent series in
$t^{-1}$, and we can write
\begin{equation}
\label{1.8.1}
\hbox{Tr}_{A}(g,t) = (-1)^{d}(\hbox{hdet } g)^{-1} t^{-l} +
\hbox{ lower degree terms},
\tag{1.8.1}
\end{equation}
where $d$ and $l$ are as in Definition \ref{xxdefn1.5}(d).
By \cite[Theorem 3.3]{JoZ} if $G$ is a finite group of
graded automorphisms acting on an Artin-Schelter
Gorenstein ring $A$, and if the homological determinant
of $g$ satisfies  $\hdet g =1$ for all $g \in G$, then
the fixed subring $A^{G}$ is Artin-Schelter Gorenstein.

Let $e(t)=a_n t^n+a_{n-1}t^{n-1}+\cdots a_1 t+a_0$ be an
integral polynomial with $a_0=1$. We say $e(t)$ is a
{\it palindrome polynomial} if $a_{n-i}=a_i$ for all $i$
and a {\it skew palindrome polynomial} if $a_{n-i}=-a_i$
for all $i$. If $e(t)$ is a skew palindrome polynomial,
then $e(1)=0$. Any polynomial which is a product of
cyclotomic polynomials is either a palindrome or a skew
palindrome polynomial.

\begin{lemma}
\label{xxlem1.9}
Let $e(t)$ be a palindrome polynomial of
degree $n$. Then $e'(1)=n e(1)/2$, where $e'(t)$ is the
derivative of $e(t)$.
\end{lemma}

\begin{proof} First we suppose $n$ is odd and let $m=
(n-1)/2$. Since $a_{n-i}=a_i$ for all $i$,
$$e(t)=\sum_{i=0}^n a_i t^i=\sum_{i=0}^m a_i(t^i+t^{n-i}).$$
Then
$$e(1)=\sum_{i=0}^m a_i (1+1)=2\sum_{i=0}^m a_i$$
and
$$e'(1)=\sum_{i=0}^m a_i(i 1^{i-1}+(n-i)1^{n-i-1})=
n\sum_{i=0}^m a_i={\frac{n}{2}}\; 2 \sum_{i=0}^m a_i=
{\frac{n}{2}}\; e(1).$$

If $n$ is even, let $f(t)=e(t)(1+t)$. Then $f(t)$ is a
palindrome polynomial of even degree. By the above proof,
the assertion holds for $f(t)$. Using $f(t)=e(t)(1+t)$ we
see that
$$e'(1)2+e(1)=f'(1)={\frac{n+1}{2}} \; f(1)
=(n+1) e(1),$$
which implies that $e'(1)=n e(1)/2$.
\end{proof}

The following two lemmas are well-known. We say a subring
$B$ of $A$ is {\it cofinite} if $A_B$ and $_BA$
are finite $B$-modules.

\begin{lemma}
\label{xxlem1.10}
Suppose $A$ is a graded algebra of finite global dimension
and $B$ is a graded subring of $A$.
\begin{enumerate}
\item
If $_BA$ is free, then $B$ has
finite global dimension. If $_BA$ is finitely
generated, then $\gldim A=\gldim B$.
\item
If $\gldim B=\gldim A$ and $_BA$ is finitely generated,
then $_BA$ is free.
\item
Suppose $B$ is a cofinite subring of $A$ with $\gldim B<\infty$.
If $A=B\oplus C$ as $B$-bimodule,
then $A$ is regular if and only if $B$ is.
\item
If $A$ is noetherian and  regular and $B$
is a factor ring of $A$ with finite global dimension,
then $B$ is regular.
\item
If  $\deg y=1$, then $A$ is noetherian and regular if and
only if $A/(y)$ is.
\end{enumerate}
\end{lemma}

Here is a list of well-known facts about
fixed subrings.

\begin{lemma}
\label{xxlem1.11}
Let $A$ be a noetherian connected graded algebra
and let $G$ be a finite subgroup of $\Aut(A)$.
\begin{enumerate}
\item
\cite[Corollaries 1.12 and 5.9]{Mon1}
$A^G$ is noetherian and $A$ is finite over $A^G$
on the left and on the right. As a consequence, $\GKdim
A=\GKdim A^G$.
\item
\cite[Corollary 1.12]{Mon1}
$A=A^G\oplus C$ as $A^G$-bimodule.
\item
(Molien's theorem) \cite[Lemma 5.2]{JiZ}
$$H_{A^G}(t)={\frac{1}{|G|}}\sum_{g\in G}
Tr_A(g,t).$$
\item
Suppose $A$ and $A^G$ are both regular, then
$\gldim A=\gldim A^G$ and $A$ is free over $A^G$
on both sides.
\end{enumerate}
\end{lemma}

\begin{proof} (d) For any noetherian regular algebra
$A$ of global dimension $d$, we have ${\rm{cd}} (A)=d-1$
\cite[Theorem 8.1(4)]{AZ}, where
$${\rm{cd}}(A)={\rm{cd}}(\Proj A)=\max \{i\;|\;
{\underline{\rm{H}}}^i_{\Proj A}
(A)\neq 0\}$$
(see \cite[p.272 and p.276]{AZ} for the definitions).
By \cite[Corollary 8.4(1)]{AZ},
${\rm{cd}}(A)\leq {\rm{cd}}(A^G)$. Since $A=A^G\oplus C$
as $A^G$-bimodule (see part (b)),
$$
\begin{aligned}
{\rm{cd}}(A)
&=\max\{i\;|\; {\underline{\rm{H}}}^i_{\Proj A} (A)\neq 0\}\\
&=\max\{i\;|\; {\underline{\rm{H}}}^i_{\Proj A} (A^G\oplus C)\neq 0\}\\
&=\max\{i\;|\; {\underline{\rm{H}}}^i_{\Proj A^G} (A^G\oplus C)\neq 0\}
\qquad {\rm [AZ, Theorem~8.3(3)]}\\
&\geq \max\{i\;|\; {\underline{\rm{H}}}^i_{\Proj A^G} (A^G)\neq 0\}\\
&={\rm{cd}} (A^G).
\end{aligned}
$$
Therefore ${\rm{cd}}(A)={\rm{cd}}(A^G)$ and
$$\gldim A={\rm{cd}}(A)+1={\rm{cd}}(A^G)+1=\gldim A^G.$$
The rest follows from Lemma \ref{xxlem1.10}(b).
\end{proof}

\begin{definition}
\label{xxdefn1.12} Let $A$ be a connected graded
algebra. If $A$ is a noetherian, regular graded
domain of global dimension $n$ and $H_A(t)=(1-t)^{-n}$,
then we call $A$ {\it a quantum polynomial ring}
of dimension $n$.
\end{definition}

By \cite[Theorem 5.11]{Sm2}, a quantum polynomial ring is
Koszul and hence it is generated in degree 1. The
GK-dimension of a quantum polynomial ring of global
dimension $n$ is $n$. In general if $A$ is finitely
generated and $H_A(t)=((1-t)^np(t))^{-1}$ for some polynomial
$p(t)$ with $p(1)\neq 0$, then the GK-dimension of $A$ is
equal to $n$. A quantum polynomial ring of dimension 2
is isomorphic to either:

(i) $k_{q}[x,y]:=k\langle x,y\rangle/(xy-qyx)$ for some
$0\neq q\in k$, or

(ii) $k_{J}[x,y]:=k\langle x,y\rangle/(xy-yx-x^2)$.

\noindent
Quantum polynomial rings of dimension 3 were classified in
\cite{ASc,ATV1}. There are many examples of quantum
polynomial rings of higher dimensions, but their
classification has not been completed yet.

\section{Quasi-reflections}
\label{sec2}

The Shephard-Todd-Chevalley theorem suggests that if a
fixed subring $A^G$ of a regular algebra $A$ is
still regular, then $G$ is some kind of a reflection
group. In the commutative case the reflection is defined
on the generating space of $k[V]$. In the noncommutative
case, this becomes a complicated issue as many examples
indicate. The following easy fact (see \cite[(1-1)]{JiZ})
suggests one possible definition of reflection.

\begin{lemma}
\label{xxlem2.1}
Let $V$ be a vector space of dimension $n$. Let $g$ be
a linear transformation of $V$ of finite order, extending
to an algebra automorphism of $A:=k[V]$. Then $g$ is a
reflection of $V$ if and only if there is $\xi \in k$ with
$$Tr_A(g,t)={\frac{1}{(1-t)^{n-1}(1-\xi t)}}.$$
\end{lemma}

By Lemma \ref{xxlem1.7} and Proposition \ref{xxprop1.8},
we have seen that if $g \neq 1$ is a finite order graded
automorphism of a regular algebra $A$, then  the
order of the pole at $t=1$ in $Tr_A(g,t)$ must be strictly
less than the order of the pole at $t=1$ in $H_A(t)$, which
is the $\GKdim A$.  We will call those graded automorphisms
whose trace has a pole at $t=1$ of order $\GKdim A -1$
quasi-reflections.

\begin{definition}
\label{xxdefn2.2}
Let $A$ be a regular graded algebra such that
$$H_A(t) =\frac{1}{(1-t)^{n}p(t)}$$
where $p(1)\neq 0$.
Let $g$ be a graded algebra automorphism of $A$.
We say that $g$ is a {\it quasi-reflection} of $A$ if
$$Tr_A(g,t)={\frac{1}{(1-t)^{n-1}q(t)}}$$
for $q(1) \neq 0$. If $A$ is a quantum polynomial ring,
then $H_A(t)=(1-t)^{-n}$. In this case $g$ is a
{\it quasi-reflection} if and only if
$$Tr_A(g,t)={\frac{1}{(1-t)^{n-1}(1-\xi t)}}$$
for some $\xi \neq 1$. (Note that we have chosen not
to call the identity map a quasi-reflection).
\end{definition}

The next example shows that if we use the definition of
a ``reflection" from the commutative case then the
condition that $G$ is generated by ``reflections" is
neither necessary nor sufficient for the fixed subring
of a noncommutative quantum polynomial ring to be regular.

\begin{example}
\label{xxex2.3}
Let $A$ be the regular algebra $k\langle x, y
\rangle/(x^2-y^2)$. This is a quantum polynomial ring
and is isomorphic to $k_{-1}[b_1,b_2]$.

(a) Let $h$ be the automorphism of  $A$ determined by
$$h(x)=-x, \quad\text{and}\quad h(y)=y.$$
Hence $h$ is a reflection of the generating space
$A_1:=kx\oplus ky$. Since $A$ has a $k$-linear basis
\begin{equation}
\label{2.3.1}
\{(yx)^i y^j\;|\;
i,j\geq 0\}\cup \{ x(yx)^i y^j\;|\; i,j\geq 0\},
\tag{2.3.1}
\end{equation}
we can compute the trace easily:
$$Tr_A(h,t)=\frac{1}{1+t^2},\quad
Tr_A(Id,t)=\frac{1}{(1-t)^2}.$$
By definition, $h$ is not a quasi-reflection.
Furthermore, the fixed subring $A^h$ is not
regular because its Hilbert series is
$$H_{A^h}(t)=\frac{1}{2}(Tr_A(h,t)+Tr_A(Id,T))=\frac{1-t+t^2}
{(1-t)^2(1+t^2)}\neq {\frac{1}{p(t)}}.$$
However, by \cite[Theorem 6.4 or Theorem 3.3]{JoZ},
$A^h$ is Artin-Schelter Gorenstein.

To summarize, there is an automorphism $h$ of $A$ with order 2
such that $h|_{A_1}$ is a reflection, but $h$ is not a
quasi-reflection and the fixed subring $A^h$ has infinite
global dimension. Consequently, $A^h\not\cong A$.
If we believe that a reflection of $A$ should give
rise to a regular fixed subring as in the
Shephard-Todd-Chevalley theorem, then we should not think
of $h$ as a reflection of $A$.

(b) Let $g$ be the automorphism of $A$ determined by
$$g(x)=i x, \quad\text{and}\quad g(y)=- i y.$$
Hence $g|_{A_1}$ is not a reflection (and neither is $g^2|_{A_1}$).
Using the $k$-linear basis in (2.3.1),
we can compute the trace easily:
$$Tr_A(g,t)=Tr(g^3,t)=\frac{1}{1-t^2},\quad Tr_A(g^2,t)=
\frac{1}{(1+t)^2}, \quad Tr_A(Id,t)=\frac{1}{(1-t)^2}.$$
So $g$ is a quasi-reflection, but $g^2$ is not.

Using the $k$-linear basis above again, one can check that
$A^g=k[xy,yx]\cong k[t,s]$. Hence $A^g$ is regular
(although not isomorphic to $A$). But $A^{g^2}$ is
not regular by a Hilbert series computation.

To summarize, there is a quasi-reflection $g$ such
that $g|_{A_1}$ is not a reflection. Since the fixed
subring $A^g$ is regular, we should think $g$ as
a reflection. On the other hand, the automorphism
$g^2$ is not a quasi-reflection and $A^{g^{2}}$ is not
regular. So we should not think $g^2$ as a reflection.
This phenomenon is very quite different from
the commutative case (where the square of a reflection
is a reflection), and it conflicts with our intuition.
\end{example}

Next we prove some general results relating quasi-reflections
to the regularity of the fixed rings.  The theorem below
justifies our definition of quasi-reflection.

\begin{theorem}
\label{xxthm2.4}
Let $A$ be noetherian and regular.
Let $G$ be a finite subgroup of $\Aut(A)$. If $A^G$ has
finite global dimension, then $G$ contains a quasi-reflection.
\end{theorem}

\begin{proof} We show that the assumption that $G$ does
not contain a quasi-reflection leads to a contradiction.

Since $A$ is regular, the Hilbert series of $A$ is
$$H_A(t)=\frac{1}{(1-t)^{n}p(t)}$$
with $p(1)\neq 0$, where $n=\GKdim A$.

By Lemma \ref{xxlem1.11}(a), $A^G$ is noetherian and $A$ is finite
over $A^G$ on the left and the right, and $\GKdim A=\GKdim A^G$.
Since $A^G$ has finite global dimension, the Hilbert series of
$A^G$ is of the form
$$H_{A^G}(t)={\frac{1}{e(t)}}=
{\frac{1}{(1-t)^n q(t)}}.$$ By Lemma \ref{xxlem1.11}(d),
$A$ is free finite over $A^G$. Hence $H_A(t)=f(t) H_{A^G}(t)$
for some polynomial $f(t)$ with non-negative integer coefficients.
Clearly $q(t)=p(t)f(t)$. Let $m=\deg p(t)$ and $n=\deg q(t)$.
Then $n-m=\deg f(t)>0$.

Expanding $H_A(t)$ into a Laurent series about $t=1$ we have
$$H_A(t)={\frac{1}{(1-t)^n}}\;{\frac{1}{p(1)}}+
{\frac{1}{(1-t)^{n-1}}}\;{\frac{p'(1)}{p^2(1)}}+\cdots
\text{higher degree terms}.$$
Similarly,
$$H_{A^G}(t)={\frac{1}{(1-t)^n}}\;{\frac{1}{q(1)}}+
{\frac{1}{(1-t)^{n-1}}}\;{\frac{q'(1)}{q^2(1)}}+\cdots
\text{higher degree terms}.$$ From Molien's theorem [Lemma
\ref{xxlem1.11}(c)], we have that
$$H_{A^G}(t)={\frac{1}{|G|}}\sum_{g\in G} Tr_A(g,t).$$
If we expand this expression into a Laurent series around
$t=1$, since $G$ does not contain any quasi-reflections,
by Lemma \ref{xxlem1.7} and Proposition \ref{xxprop1.8}
the Laurent series of $Tr(g,t)$ has lowest possible degree term
$(1-t)^{-(n-2)}$. Hence the first terms of the
sum come entirely from the trace of the identity
$Tr_A(Id_A,t)=H_A(t)$. Hence
$$H_{A^G}(t)={\frac{1}{|G|}}
\big[ {\frac{1}{(1-t)^n}}\;{\frac{1}{p(1)}}+
{\frac{1}{(1-t)^{n-1}}}\;{\frac{p'(1)}{p^2(1)}}+\cdots
\text{higher degree terms}\big].$$
Equating coefficients in the two expressions for $H_{A^G}(t)$
we have that
$$q(1)=|G|p(1),\quad \text{and}\quad {\frac{q'(1)}{q(1)^2}}
={\frac{1}{|G|}}\; {\frac{p'(1)}{p(1)^2}}.$$
Since $p(t)$ and $q(t)$ are products of cyclotomic
polynomials, they are palindrome polynomials. By Lemma \ref{xxlem1.9},
$$2p'(1)=mp(1),\quad \text{and}\quad 2q'(1)=nq(1).$$
Hence we have
\[\frac{q'(1)}{(q(1))^2} = \frac{n}{2q(1)} \;\; \text{  and  } \;\; \frac{p'(1)}{|G| (p(1))^2}
= \frac{m}{2|G|(p(1))}  = \frac{m}{2q(1)},\] and so
\[\frac{n}{2q(1)} = \frac{m}{2q(1)}\]
gives $n = m$, a contradiction.
\end{proof}

The number of quasi-reflections in $G$ can also be computed.

\begin{theorem}
\label{xxthm2.5}
Suppose $A$ is a quantum polynomial ring, and
let $G$ be a finite
subgroup of $\Aut(A)$. Denote the number of quasi-reflections
in $G$ by $r$.
\begin{enumerate}
\item
If $H_{A^G}(t)$ is expanded into a Laurent series around $t=1$,
then the coefficient of $(1-t)^{-(n-1)}$ is given by
$r/2|G|$.
\item
Suppose $A^G$ is regular and $H_{A^G}(t)=((1-t)^{n}q(t))^{-1}$.
Then $q(1)=|G|$ and $r=\deg q(t)$.
\end{enumerate}
\end{theorem}

\begin{proof}
(a) Let $g_1,\cdots,g_r$ be the quasi-reflections (that are
not the identity) in $G$, and let $h_1,\cdots,h_s$ be the
non-identity elements of $G$ that are not quasi-reflections.
By Lemma \ref{xxlem1.6}, for all $g\in G$, $Tr_A(g,t)=1/e_g(t)$
where $e_g(t)$ has degree $n$. Suppose now $g$ is a
quasi-reflection. Then
$$Tr_A(g,t)={\frac{1}{(1-t)^{n-1}(1-\lambda t)}}$$
where $\lambda \neq 1 \in k$. By Lemma \ref{xxlem1.6}(d),
$\lambda$ is a root of unity. Thus the Laurent
expansion of $Tr_A(g,t)$ around $t=1$ is given by
$$Tr_A(g,t)={\frac{1}{(1-t)^{n-1}}}
\big[{\frac{1}{1-\lambda}}+(1-t)a_1+\cdots \big].$$
By Lemma \ref{xxlem1.4}, the Laurent expansion of
$Tr_A(g^{-1},t)$ is given by
$$Tr_A(g^{-1},t)=\overline{Tr_A(g,t)}={\frac{1}{(1-t)^{n-1}}}
\big[{\frac{1}{1-\overline{\lambda}}}+(1-t)\overline{a_1}+\cdots\big].$$
In particular, $g^{-1}$ is again a quasi-reflection. This
also shows that if $g$ has order $2$, then
$$Tr_A(g,t)={\frac{1}{(1-t)^{n-1}(1+t)}}
={\frac{1}{(1-t)^{n-1}}}
\big[{\frac{1}{2}}+(1-t)a_1+\cdots \big].$$
Note that
$${\frac{1}{1-\lambda}}+{\frac{1}{1-\overline{\lambda}}}
={\frac{1-\overline{\lambda}+1-\lambda}
{(1-\lambda)(1-\overline{\lambda})}}
={\frac{1-\overline{\lambda}+1-\lambda}
{1-\lambda-\overline{\lambda}+\lambda\overline{\lambda}}}
=1$$
since $\lambda\overline{\lambda}=1$. Now let $h$ be a
non-identity element in $G$ that is not a quasi-reflection.
Then the Laurent expansion of its trace is of the form
$$Tr_A(h,t)={\frac{1}{(1-t)^{n-2}}}
(c_0+c_1(1-t)+\cdots ).$$
By Molien's theorem [Lemma \ref{xxlem1.11}(c)] we have
\begin{equation}
\label{2.5.1}
H_{A^G}(t)
\;=\;{\frac{1}{|G|}}\sum_{g\in G}Tr_A(g,t)
\;=\;{\frac{1}{|G|}}
[{\frac{1}{(1-t)^n}}+\sum_{i=1}^r Tr_A(g_i,t)+\sum_{j=1}^s Tr_A(h_j,t)].
\tag{2.5.1}
\end{equation}
We see that the only contributions to the coefficient of
$\frac{1}{(1-t)^{n-1}}$ come from the
$\sum_{i=1}^r Tr_A(g_i,t)$ term. By grouping each $g_i$ with
its inverse, we see that the coefficient is exactly $r/2|G|$.
(b) Expanding $H_{A^G}(t)$ around $t=1$, we have
\begin{equation}
\label{2.5.2}
H_{A^G}(t)={\frac{1}{(1-t)^n}}\;{\frac{1}{q(1)}}+
{\frac{1}{(1-t)^{n-1}}}\;{\frac{q'(1)}{q^2(1)}}+\cdots .
\tag{2.5.2}
\end{equation}
Comparing the coefficients of $1/(1-t)^n$ and $1/(1-t)^{n-1}$
in \eqref{2.5.1} and \eqref{2.5.2}, we see that
 $${\frac{1}{|G|}}\;=\;{\frac{1}{q(1)}},\quad \text{and}\quad
{\frac{r}{2|G|}}\;=\;{\frac{q'(1)}{q^2(1)}}.$$
Combining with Lemma \ref{xxlem1.9}, we obtain $r=\deg q(t)$.
\end{proof}
The following lemma will be used in the next section.

\begin{lemma} \cite[Theorem 2.3.2]{JiZ}
\label{xxlem2.6}
Let $A$ be a noetherian regular algebra
and let $g$ be a graded algebra automorphism of $A$.
Suppose $B$ is a factor ring of $A$ such that $g$ induces
an algebra automorphism $g'$ of $B$. Then $Tr_B(g',t)=p(t) Tr_A(g,t)$
where $p(t)$ is a polynomial of $t$ with $p(0)=1$.
\end{lemma}

\section{Quasi-reflections of quantum polynomial rings}
\label{sec3}

In this section we will classify all possible quasi-reflections
of a quantum polynomial ring. The proof of the following
main result requires several lemmas.

\begin{theorem}
\label{xxthm3.1} Let $A$ be a quantum polynomial ring of
global dimension $n$. If $g\in \Aut(A)$ is a quasi-reflection
of finite order, then $g$ is in one of the following two
cases:
\begin{enumerate}
\item
There is a basis of $A_1$, say $\{b_1,\cdots,b_n\}$,
such that $g(b_j)=b_j$ for all $j\geq 2$ and $g(b_1)
=\xi b_1$. Namely, $g|_{A_1}$ is a reflection.
\item
The order of $g$ is $4$ and there is a basis of $A_1$,
say $\{b_1,\cdots,b_n\}$,
such that $g(b_j)=b_j$ for all $j\geq 3$ and $g(b_1)
=i\; b_1$ and $g(b_2)=-i\; b_2$ (where $i^2=-1$).
\end{enumerate}
\end{theorem}

We start with a lemma about sums of roots of unity.

\begin{lemma}
\label{xxlem3.2} Every solution of the following system
\begin{equation}
\label{3.2.1}
n=x_1+x_2+\cdots+x_{n+2}
\tag{3.2.1}
\end{equation}
with each $x_i$ being a root of unity, but not 1, is in
one of the following cases:
\begin{enumerate}
\item
$n=0$, $x_1=\xi$ and $x_2=-\xi$ where $\xi$ is a root of 1,
which is not $\pm 1$.
\item
$n=2$, $x_1=x_2=\zeta_6$ and $x_3=x_4=\overline{\zeta_6}$
and all possible permutations.
\end{enumerate}
\end{lemma}

\begin{proof} First we claim that $x_i$ cannot be $-1$.
If $x_i=-1$, say $x_{n+2}=-1$, the equation becomes
$$n=x_1+\cdots +x_{n+1}-1$$
or
$$n+1=x_1+\cdots +x_{n+1}.$$
Since every $x_i$ is a root of 1, but not 1, the real part
of each $x_i$ is strictly less than $1$, and there is no
solution to the above equation. Thus we proved our claim
that $x_i\neq -1$ for all $i$.

Let $w_i$ be the order of $x_i$ and let $w$ be the gcd
of the $w_i$. Since $x_i\neq -1$, $w_i\geq 3$. The Galois group
of ${\mathbb Q}(\zeta_w)$ over ${\mathbb Q}$ is
$({\mathbb Z}/w{\mathbb Z})^*$. For every $p$ coprime to
$w$, let $\Xi_p$ denote the automorphism determined by
$\Xi_p(\zeta)=\zeta_w^p$. Let $\Xi$ denote the group
$\{\Xi_p\;|\; (p,w)=1\}=({\mathbb Z}/w{\mathbb Z})^*$ .
The order of $\Xi$ is $\phi(w)$. Recall that
$$\sum_{(p,w)=1} \zeta_w^p
\;=\;\mu(w)
\; :=\;
\begin{cases} (-1)^t & \text{if $w$ is a product of $t$
distinct primes}
\\ 0 & \text{if $w$ is not square-free}
\end{cases}
$$
where the notation $(p,w)=1$ means the set of integers $p$ such
that $0<p<w$ and that $p$ is coprime to $w$ (see \cite[(16.6.4),
p. 239]{HW}). Let $\Xi[x_i]$ denote $\sum_{(p,w)=1} \Xi_p(x_i)$.
Since $\Xi[x_i]$ is stable under $\Xi$-action, it contains
$\phi(w)/\phi(w_i)$ copies of each $w_i$-th primitive root of
unity. Thus we have
$$\Xi[x_i]
\;=\;
{\frac{\phi(w)}{\phi(w_i)}}\sum_{(p,w_i)=1} \zeta_{w_i}^p
\;=\;
{\frac{\phi(w)}{\phi(w_i)}}\mu(w_i).$$
Applying $\Xi$ to equation \eqref{3.2.1} we obtain
that
$$n \phi(w)=\sum_{i=1}^{n+2} {\frac{\phi(w)}{\phi(w_i)}}\mu(w_i).$$
Hence
$$n=\sum_{i=1}^{n+2} {\frac{1}{\phi(w_i)}}\mu(w_i)
=\sum_{i=1}^{n+2} ({\frac{\mu(w_i)}{\phi(w_i)}}-1)+n+2$$
or
\begin{equation}
\label{3.2.2}
\sum_{i=1}^{n+2} (1-{\frac{\mu(w_i)}{\phi(w_i)}})=2.
\tag{3.2.2}
\end{equation}
Since the M{\"o}bius function $\mu(w_i)$ is either 1, 0,
or $-1$, and $\phi(w_i)$ is at least $2$, the largest possible
$n$ is $2$. So we consider three cases:

$n=0$: If $x_1=\xi$, then $x_2=-\xi$. This is case (a).

$n=1$: If $\mu(w_3)\leq 0$, then we have
$$\sum_{i=1}^{2} (1-{\frac{\mu(w_i)}{\phi(w_i)}})\leq 1.$$
This implies that
$${\frac{\mu(w_1)}{\phi(w_1)}}={\frac{\mu(w_2)}{\phi(w_2)}}=
\frac{1}{2}, \quad \text{and}\quad \mu(w_3)=0.$$
The only possibility is $w_1=w_2=6$. Thus $x_1=\zeta_6^i$,
$x_2=\zeta_6^j$ where $i,j\in \{1,5\}$, and $x_3=\xi$ where
the order of $\xi$ is not square-free. As complex numbers,
$$x_1={\frac{1}{2}}+ a {\frac{\sqrt{3}}{2}} i, \quad
\text{and}\quad x_2={\frac{1}{2}}+ b {\frac{\sqrt{3}}{2}} i$$
where $a,b\in \{1,-1\}$. Hence
$$x_3=1-x_1-x_2=-a {\frac{\sqrt{3}}{2}} i-b {\frac{\sqrt{3}}{2}} i$$
which is clearly not a root of unity. This yields a contradiction,
so $\mu(w_3)=1$. By symmetry, $\mu(w_1)=\mu(w_2)=1$.
By \eqref{3.2.2}, we have
$$\sum_{i=1}^3 {\frac{1}{\phi(w_i)}}=1,$$
which has three solutions up to permutation:
$\{\phi(w_1),\phi(w_2),\phi(w_3)\}=\{3,3,3\}$, or
$\{2,4,4\}$, or $\{2,3,6\}$. But there is no $w$ such that
$\phi(w)=3$. Hence $\phi(w_1)=2$ and $\phi(w_2)=\phi(w_3)=4$.
Together with $\mu(w_1)=\mu(w_2)=\mu(w_3)=1$, we see that
$$w_1=6, \quad \text{and}\quad w_2=w_3= 10 .$$
With these constraints, it is straightforward to show
that there is no solution to the equation $x_1+x_2+x_3=1$.
In conclusion, there is no solution when $n=1$.

$n=2$: The equation \eqref{3.2.2} becomes
$$\sum_{i=1}^4 {\frac{\mu(w_i)}{\phi(w_i)}}=2.$$
Since $\phi(w_i)\geq 2$, then only solution is
$${\frac{\mu(w_i)}{\phi(w_i)}}={\frac{1}{2}}$$
for all $i$. Then $w_i=6$ for all $i$. Hence
$$\{x_1,x_2,x_3,x_4\}\;=\;
\{\zeta_6,\zeta_6,\zeta_6^5,\zeta_6^5\}$$
up to permutations. This is case (b).
\end{proof}

Now we can show a part of Theorem \ref{xxthm3.1}.

\begin{proposition}
\label{xxprop3.3} Suppose $A$ is a quantum polynomial ring of
global dimension $n$. If $g\in \Aut(A)$ is a
quasi-reflection of finite order, then $g$ is in one of
the following cases:
\begin{enumerate}
\item
There is a basis of $A_1$, $\{b_1,\cdots,b_n\}$
such that $g(b_i)=b_i$ for all $i\geq 2$ and $g(b_1)
=\xi b_1$.
\item
The order of $g$ is $2m$ and there is a basis of $A_1$,
$\{b_1,\cdots,b_n\}$
such that $g(b_i)=b_i$ for all $i\geq 3$ and $g(b_1)
=\xi\; b_1$ and $g(b_2)=-\xi\; b_2$.
\item
The order of $g$ is 6 and there is a basis of $A_1$,
$\{b_1,\cdots,b_n\}$ such that $g(b_i)=b_i$ for all $i\geq 4$
and $g(b_i)=\xi\; b_i$ for $i=1,2$ and $g(b_3)=\overline{\xi}\; b_3$
where $\xi=\zeta_6$ or $\zeta_6^5$.
\item
The order of $g$ is 6 and there is a basis of $A_1$,
$\{b_1,\cdots,b_n\}$ such that $g(b_i)=b_i$ for all $i\geq 5$
and $g(b_i)=\zeta_6\; b_i$ for $i=1,2$ and $g(b_j)=\zeta_6^5 \; b_j$
for $j=3,4$.
\end{enumerate}
\end{proposition}

\begin{proof} Note that the Hilbert series of $A$ is
$H_A(t)=(1-t)^{-n}$. By the definition of a
quasi-reflection, $Tr_A(g,t)=(1-t)^{-n+1}(1-\xi t)^{-1}$
for some root of unity $\xi$. By Proposition \ref{xxprop1.8},
$\xi\neq 1$. Furthermore using equation \eqref{1.8.1}
we compute $\xi=\hdet g$, so
the order of $g$ is a multiple of the order of $\xi$.

Since $g$ has a finite order, there is a basis of $A_1$,
$\{b_1,\cdots,b_n\}$, such that $g(b_i)=x_i b_i$ for all $i$,
where every $x_i$ is a root of unity whose order divides the order
of $g$.

Since the coefficient of the $t$ term in the power series
expansion of $Tr_A(g,t)$ is $tr (g|_{A_1})$, we have
$$\sum_{i=1}^n x_i= (n-1)+\xi.$$
Cancelling all $x_i$'s with $x_i=1$, and permuting
$x_i$ if necessary, we have
\begin{equation}
\label{3.3.1}
\sum_{i=1}^m x_i=(m-1)+\xi
\tag{3.3.1}
\end{equation}
where $x_i\neq 1$ for all $i=1,\cdots,m$. If $m=1$, this
is case (a), and we are done. Now assume $m\geq 2$.
First we assume $\xi\neq -1$. Moving $\xi$ to the
left-hand side, equation \eqref{3.3.1} becomes
$$\sum_{i=1}^m x_i-\xi=m-1.$$  Since $m \geq 2$
by Lemma \ref{xxlem3.2} there are two possibilities.
The first case is Lemma \ref{xxlem3.2} case (b): $m-1=2$ and
$-\xi=\zeta_6$ or $-\xi=\zeta_6^5$.
By symmetry, we may assume $-\xi=\zeta_6^5$ and
$x_1=x_2=\zeta_6,x_3=\zeta_6^5$. This is our case (c).

The second case is $m\geq 2$ and $\xi=-1$. Then
equation \eqref{3.3.1} becomes
$$\sum_{i=1}^m x_i=m-2.$$
By Lemma \ref{xxlem3.2} there are two cases. Either
$m=2$, $x_1=\xi$ and $x_2=-\xi$, which is our case (b),
or $m-2=2$, which is our case (d).
\end{proof}

In the rest of this section we will eliminate most
of cases in Proposition \ref{xxprop3.3}(b,c,d).
In some cases we will use the notion of a
${\mathbb Z}^2$-graded algebra. We say that $R$ is a
connected ${\mathbb Z}^2$-graded algebra,
if all the generators of $R$ are either in ${\mathbb Z}^{+}
\times {\mathbb Z}^{\geq 0}$ or ${\mathbb Z}^{\geq 0}
\times {\mathbb Z}^{+}$. The Hilbert series of a
${\mathbb Z}^2$-graded algebra/module $M$ is given by
$$H_M(t,s)=\sum_{i,j} \dim M_{i,j} \; t^i s^j.$$
The usual techniques for Hilbert series
of ${\mathbb Z}$-graded algebras/modules
extend to the ${\mathbb Z}^2$-graded setting.
The following lemma is clear.

\begin{lemma}
\label{xxlem3.4}
Let $R$ be a connected ${\mathbb Z}^2$-graded algebra.
\begin{enumerate}
\item
If we use assignment $\deg (1,0)=\deg(0,1)=1$ to make $R$
a ${\mathbb Z}$-graded algebra, then $H_R(t)=H_R(t,t)$.
\item
Suppose that $R$ is noetherian of finite global dimension.
Then $H_A(t,s)=(e(t,s))^{-1}$ where $e(t,s)$ is an integral
polynomial in $t,s$ with $e(0,0)=1$.
\item
Suppose $R$ is noetherian and has finite global dimension.
Let $M$ be a finitely generated ${\mathbb Z}^2$-graded $R$-module.
Then $H_M(t,s)=p(t,s)H_R(t,s)$ for some integral
polynomial $p(t,s)$.
\end{enumerate}
\end{lemma}

Next we assume that $A$ is generated by $A_1$, which
has a basis $\{b_1,\cdots,b_n\}$.

\begin{proposition}
\label{xxprop3.5}
Suppose that a quantum polynomial algebra $A$ is
${\mathbb Z}^2$-graded with
$\deg b_i=(1,0)$ for all $i=1,\cdots,m$ and $\deg b_j=(0,1)$
for all $j=m+1,\cdots,n$. Let $B$ and $C$ be graded
subalgebras generated by $\{b_1,\cdots,b_m\}$ and
$\{b_{m+1},\cdots,b_n\}$ respectively. Then
\begin{enumerate}
\item
$B\cong A/(C_{\geq 1})$ and $C\cong A/(B_{\geq 1})$.
\item
Both $B$ and $C$ are quantum polynomial rings.
\item
If $m=1$, then $A$ is the Ore extension $C[b_1;\sigma]$
for some graded algebra automorphism $\sigma$ of $C$.
As a consequence, $b_1$ is a normal element of $A$.
\end{enumerate}
\end{proposition}

\begin{proof} For any ${\mathbb Z}^2$-graded module
$M$, define
$$M^{{\mathbb Z}\times 0}=\{x\in M\;|\; \deg x\in
{\mathbb Z}\times 0\}.$$
Similarly we define $M^{0\times {\mathbb Z}}$.
Note that $B=A^{{\mathbb Z}\times 0}$
and $C=A^{0\times {\mathbb Z}}$. Hence $B$ and $C$
are noetherian. In the following proof, we deal only with
$B$. By symmetry, the assertions hold for $C$ also.

(a) There is a natural map $B\to A\to A/(C_{\geq 1})$.
Clearly this is a surjection. For every $x\in B$
$\deg x \in {\mathbb Z}\times 0$. For every $y\in (C_{\geq 1})$,
$\deg y\in {\mathbb Z}\times {\mathbb Z}^{+}$.
Hence $B\cap (C_{\geq 1})=0$ and the map $B\to A/(C_{\geq 1})$
is injective.

(b) First we prove that $B$
has finite global dimension. Take a graded free resolution
of the trivial $A$-module $k$:
\begin{equation}
\label{3.5.1}
0\to P_n\to \cdots P_1\to A\to k\to 0
\tag{3.5.1}
\end{equation}
where each $P_i$ is a direct sum of $A[-v,-w]$ for
some $v,w\geq 0$. Then we have  a resolution of
$B$-modules:
\begin{equation}
\label{3.5.2}
0\to P_n^{{\mathbb Z}\times 0}\to \cdots
P_1^{{\mathbb Z}\times 0} \to A^{{\mathbb Z}\times 0}\to k\to 0.
\tag{3.5.2}
\end{equation}
We claim that $P_i^{{\mathbb Z}\times 0}$ is a
free $B$-module for every $i$. It suffices to show that
each $A[-v,-w]^{{\mathbb Z}\times 0}$ is either 0 or
a shift of $B$. It is clear from the definition
that
$$A[-v,-w]^{{\mathbb Z}\times 0}=
\begin{cases} 0 & w>0\\ B[-v] & w=0\end{cases}.$$
So the trivial $B$-module $k$ has a finite
free resolution, and $B$ has finite global
dimension. By (a) and \cite[Corollary 8.4]{AZ},
$B$ satisfies the $\chi$-condition.  Since
$\uExt_B^i(k,B)_j$ is finite dimensional
for all $j$, and the $\chi$-condition implies that
$\uExt_B^i(k,B)_j$ is bounded, then it
follows that $\uExt_B^i(k,B)$ is finite
dimensional. From \cite[Theorem 1.2]{Z} it then
follows that $B$ satisfies the Artin-Schelter Gorenstein
condition, and hence $B$ is regular. Clearly $B$
is a domain.

Next we study the Hilbert series of $B$. Let $H_B(t)=
a(t)^{-1}$ and $H_C(s)=b(s)^{-1}$. By Lemma \ref{xxlem3.4}(c),
there are $p(t,s)$ and $q(t,s)$ such that $a(t)p(t,s)=e(t,s)$
and $b(s)q(t,s)=e(t,s)$. Set $t=s$, we have $e(t,t)=
(1-t)^{-n}$. Then $a(t)=(1-t)^a$ and $b(s)=(1-s)^{b}$
for some integers $a,b$, and $e(t,s)=(1-t)^a (1-s)^b r(t,s)$.
Since $B$ is generated by $m$ elements and $C$ is
generated by $n-m$ elements, $a=m$ and $b=n-m$.
Thus $e(t,s)=(1-t)^a (1-s)^b$. Since the resolution
\eqref{3.5.1} is Koszul, after converting to the ${\mathbb Z}$-grading
the resolution \eqref{3.5.2} is also Koszul. So the global
dimension of $B$ is $m$. Thus we have proved (b).

(c) By (a) with $m=1$, $C=A/(b_1)$, which has Hilbert series
$(1-t)^{-n+1}$. This implies that the Hilbert series
of the ideal $(b_1)$ is $t(1-t)^{-n}$. Since $A$ is a
domain, the Hilbert series of $b_1 A$ and $Ab_1$ are equal
to $t(1-t)^{-n}$. Thus $b_1 A=(b_1)=Ab_1$, and $b_1$ is a
normal element of $A$. Since $A/(C_{\geq 1})=k[b_1]$,
$b_1^2$ will not appear in any of the relations of $A$.
Thus the number of the relations between $b_1$ and
$b_j$ for $j\geq 2$ is $n-1$. The only relations between $b_1$
and $b_j$ are relations that can be
written as, for every $j=2,\cdots,n$,
$$b_j b_1 =b_1\sigma(b_j)$$
for some $\sigma(b_j)\in C_1$. Since $b_1C=Cb_1$,
$\sigma$ extends to an algebra automorphism of $C$.
Therefore $A=C[b_1;\sigma]$.
\end{proof}

\begin{lemma}
\label{xxlem3.6} Let $g$ be a quasi-reflection described
in Proposition \ref{xxprop3.3}(b). Then the order of $g$
is 4 and $\xi=i$.
\end{lemma}

\begin{proof} We have seen that this situation can occur
(Example \ref{xxex2.3}(b)).
If $g$ has order $4$, then this is the only
solution up to a permutation. We now assume the order
of $g$ is not 4 and produce a contradiction.

Clearly the order of $g$ is not $2$. Hence the order of $g$
is at least $6$ and the order of $\xi$ is not
$4$. If $r:=\sum a_{ij} b_i b_j=0$ is a relation
of $A$, then after applying $g$ we have
$$g(r):=\sum_{i,j\geq 3}a_{ij}b_i b_j+
\xi^2 (a_{11}b_1^2+a_{22}b_2^2)-\xi^2(a_{12}b_1b_2+a_{21}b_2b_1)$$
$$+\xi (\sum_{i\geq 3} (a_{1i}b_1b_i+a_{i1}b_ib_1))
-\xi (\sum_{i\geq 3} (a_{2i}b_2b_i+a_{i2}b_ib_2))
=0.$$
We obtain similar expressions for $g^p(r)$ for $p=0,1,2,3,4$,
which gives rise a system of equations
$$Y_1+(\xi^2)^0 Y_2+(-\xi^2)^0 Y_3+\xi^0 Y_4+ (-\xi)^0 Y_5=0$$
$$Y_1+(\xi^2)^1 Y_2+(-\xi^2)^1 Y_3+\xi^1 Y_4+ (-\xi)^1 Y_5=0$$
$$Y_1+(\xi^2)^2 Y_2+(-\xi^2)^2 Y_3+\xi^2 Y_4+ (-\xi)^2 Y_5=0$$
$$Y_1+(\xi^2)^3 Y_2+(-\xi^2)^3 Y_3+\xi^3 Y_4+ (-\xi)^3 Y_5=0$$
$$Y_1+(\xi^2)^4 Y_2+(-\xi^2)^4 Y_3+\xi^4 Y_4+ (-\xi)^4 Y_5=0$$
where $Y_1=\sum_{i,j\geq 3}a_{ij}b_i b_j$,
$Y_2=a_{11}b_1^2+a_{22}b_2^2$, $Y_3=a_{12}b_1b_2+a_{21}b_2b_1$,
$Y_4= \sum_{i\geq 3} (a_{1i}b_1b_i+a_{i1}b_ib_1)$, and $Y_5=
\sum_{i\geq 3} (a_{2i}b_2b_i+a_{i2}b_ib_2)$. It is easy to check
that the determinant of the coefficients in the above system is nonzero
when $\xi^4\neq 1$. Hence $Y_i=0$ for all $i=1,2,3,4,5$.
This means that $A$ is ${\mathbb Z}^2$-graded when we assign
$\deg b_1=\deg b_2=(1,0)$ and $\deg b_i=(0,1)$ for
all $i\geq 3$. By Proposition \ref{xxprop3.5}, the subalgebra $B$
generated by $b_1$ and $b_2$ is a quantum polynomial ring.
So $B$ has only one relation. Let $g'$ be the
automorphism of $B$ induced  by $g$. By Lemma \ref{xxlem2.6},
$Tr_B(g,t)=p(t) Tr_A(g,t)$. Since $g$ is a quasi-reflection,
so is $g'$. It suffices to show there is no quasi-reflection
$g'$ of order larger than $4$. The unique relation of $B$ is
either $b_1^2+b_2^2=0$ or $b_1b_2+q b_2b_1=0$ (for $q \neq 0$),
up to  a linear
transformation. In both cases, $Tr_B(g',t)$ is easy to compute:

If $b_1b_2+q b_2b_1=0$, then $Tr_B(g',t)=[(1-\xi t)(1+\xi t)]^{-1}$.

If $b_1^2+b_2^2=0$, then $Tr_B(g',t)=(1+\xi^2 t^2)^{-1}$.

In each of these cases $g'$ is not a quasi-reflection. Therefore the only
possibility is that the order of $g$ is $4$.
\end{proof}

\begin{lemma}
\label{xxlem3.7}
Let $A$ be a graded domain generated by two elements.
\begin{enumerate}
\item
If $A$ has at least one quadratic relation, then $A$ is
a quantum polynomial ring, namely $A$ is isomorphic to either
$k_q[b_1,b_2]$ or $k_J[b_1,b_2]$.
\item
If $A$ is a quadratic algebra of finite GK-dimension, then
$A$ is a quantum polynomial ring.
\end{enumerate}
\end{lemma}

\begin{proof} (a) Let $r:=\sum_{i,j}a_{ij}b_ib_j=0$
be one of the relations. Since $A$ is a domain this
relation is not a product of two linear terms.
Then, possibly after a field extension,
$B:=k\langle b_1,b_2\rangle/(r)$ is a regular algebra of
dimension 2 (see \cite[p. 1601]{StZ}), and hence is
isomorphic to either $k_q[b_1,b_2]$ or $k_J[b_1,b_2]$.
In either case, one can check that every homogeneous
element in $B$ is a product of linear terms, and thus any
proper graded factor ring of $B$ will not be a domain.
Therefore $A=B$.

(b) Since $A$ has finite GK-dimension, $A$ cannot be a free
algebra. So $A$ has at least one quadratic relation, and
the assertion follows from (a).
\end{proof}

\begin{lemma}
\label{xxlem3.8}
Let $A$ be a quantum polynomial ring and let $\{b_1,\cdots,b_n\}$
be a $k$-linear basis of $A_1$. Suppose $g$ is in $\Aut(A)$
such that
$$g(b_1)=\xi b_1,\; g(b_2)=\xi b_2,\; g(b_j)=\xi^{-1} b_j,
\;\text{and} \; g(b_i)=b_i$$
for all $3\leq j<m$ and $m\leq i$. Suppose $\xi^4\neq 1$
and $\xi^3 \neq 1$. Let $B$ be the subalgebra generated by
$b_1$ and $b_2$. Then $B$ is a quantum polynomial ring and
$B\cong A/(b_s,s\geq 3)$.
\end{lemma}

\begin{proof} Since $A$ is a quadratic algebra and
$\xi^4\neq 1$, the relations in $A$ will be homogeneous
with respect to the grading
$$\deg(b_1)=\deg (b_2)=(1,1),\; \deg(b_j)=(1,-1),\;
\text{and} \; \deg (b_i)=(1,0)$$
where $3\leq j<m$ and $m\leq i$. Hence
$A$ is a ${\mathbb Z}^2$-graded algebra (different from
the one in Lemma \ref{xxlem3.4}). Any relation
in $B$ has degree $(n,n)$, but any relations
involving $b_s$ for $s\geq 3$ has degree $(n,m)$
for $m<n$. Thus the canonical map
$$B\to A\to A/(b_s;s\geq 3)$$
is an isomorphism. Since $B$ is a quadratic domain
of finite GK-dimension, by Lemma \ref{xxlem3.7},
it is a quantum polynomial ring.
\end{proof}

Now we are ready to prove Theorem \ref{xxthm3.1}.

\begin{proof}[Proof of Theorem \ref{xxthm3.1}]
It remains to show that there is no quasi-reflection as
described in Proposition \ref{xxprop3.3}(c,d). The proofs
are very similar for cases (c) and (d), so we work on only
case (c).

Suppose that $g$ as described in Proposition
\ref{xxprop3.3}(c) exists. Here $\xi=\zeta_6$, and so
$\xi^4\neq 1$ and $\xi^3 \neq 1$, and the
hypotheses of Lemma \ref{xxlem3.8} are satisfied.
Thus $B$ is a quantum polynomial ring such that
$B=A/(b_s;s\geq 3)$. When restricted to $B$, $g$ is equal
to $\xi Id_B$, and thus $Tr_B(g,t)=(1-\xi t)^{-2}$.
By Lemma \ref{xxlem2.6},
$${\frac{1}{(1-\xi t)^2}}=Tr_B(g,t)=p(t)Tr_A(g,t)
=p(t){\frac{1}{(1-t)^{n-1}(1-\xi' t)}}.$$
Since $p(t)$ is a polynomial, we have
$$p(t)(1-\xi t)^2=(1-t)^{n-1}(1-\xi' t),$$
which is impossible.
\end{proof}

If a quasi-reflection is as described in Theorem
\ref{xxthm3.1}(a), then it is like a classical
reflection. The quasi-reflection in Theorem
\ref{xxthm3.1}(b) is very mysterious and deserves
further study. The following definition seems
sensible, at least for quantum polynomial rings.

\begin{definition}
\label{xxdefn3.9}
Let $A$ be a quantum polynomial ring.
\begin{enumerate}
\item
A quasi-reflection $g$ of $A$ is called {\it reflection}
if $g|_{A_1}$ is a reflection.
\item
A quasi-reflection $g$ of $A$ is called {\it mystic
reflection} if $g|_{A_1}$ is not a reflection.
\end{enumerate}
\end{definition}

\section{Mystic reflections of quantum polynomial rings}
\label{sec4}

In this section we focus on the mystic reflections of
quantum polynomial rings. We will see that all
mystic reflections are similar to the automorphism $g$
in Example \ref{xxex2.3}(b).

First we state a lemma that we will use in this analysis;  its
proof is similar to that of Lemma \ref{xxlem3.2}, but there are many
cases, and some require numerical approximations from
Maple, and hence we state it without proof.
Let $\zeta_k$ be the primitive $k$th root of unity given by
$\zeta_k = e^{\frac{2 \pi i}{k}}$.

\begin{lemma}
\label{xxlem4.1}
Consider the system
$$n =  x_1+x_2 + \cdots + x_{n+4}$$
where $n$ is a nonnegative integer and each $x_i$ is a root of
unity not equal to $1.$  Then $0 \leq n \leq 4$ and the solutions
fall into the following cases:
\begin{enumerate}
\item[(1)]
If at least one $x_i$ is equal to $-1$, then we are in the
situation of Lemma \ref{xxlem3.2}.
\item[(2)]
If $x_i+x_j =0,$ then again we are in the situation of
Lemma \ref{xxlem3.2}. In particular if $n = 0,$ then all solutions
are of the form $\xi - \xi + \mu - \mu = 0$
for roots of unity $\xi$ and $\mu$.
\end{enumerate}

For the remainder suppose that neither $(1)$ nor $(2)$ holds.

\begin{enumerate}
\item[(3)]
If $n=1$, then the solutions are given by
    \begin{itemize}
    \item[(a)]
    $(\zeta_6 + \zeta_6^5) + \xi(1 + \zeta_3 + \zeta_3^2) = 1,$
    where $\xi$ is an arbitrary root of unity;
    \item[(b)]
    $(\zeta_{10} +\zeta_{10}^3+\zeta_{10}^7)
          + (\zeta_{15} + \zeta_{15}^{11}) =1;$
    \item[(c)]
    $(\zeta_{10} +\zeta_{10}^3+\zeta_{10}^9)
          + (\zeta_{15}^8 + \zeta_{15}^{13}) =1;$
    \item[(d)]
    $(\zeta_{10} +\zeta_{10}^7+\zeta_{10}^9)
          + (\zeta_{15}^2 + \zeta_{15}^{7}) =1;$
    \item[(e)]
    $(\zeta_{10}^3 +\zeta_{10}^7+\zeta_{10}^9)
          + (\zeta_{15}^4 + \zeta_{15}^{14}) =1.$
    \end{itemize}
\item[(4)]
If $n = 2,$ then
$(\zeta_6 +\zeta_6^5) + (\zeta_{10} + \zeta_{10}^3
  + \zeta_{10}^7 +\zeta_{10}^9) = 2$
is the only solution.
\item[(5)]
If $n=3,$ then there is no solution.
\item[(6)]
If $n=4$, then $4(\zeta_6+ \zeta_6^5) = 4$ is the only solution.
\end{enumerate}
\end{lemma}

Next we classify the mystic reflections of a quantum
polynomial ring.

\begin{lemma}
\label{xxlem4.2}
Let $g$ be a mystic reflection of a quantum polynomial ring
$A$ of global dimension $n$. Then the order of $g$ is 4 and
$$Tr_A(g,t)=Tr_A(g^3,t)={\frac{1}{(1-t)^{n-1}(1+t)}},\;
\quad
Tr_A(g^2,t)={\frac{1}{(1-t)^{n-2}(1+t)^2}}.$$
\end{lemma}

\begin{proof} The order of $g$ is 4 by Theorem
\ref{xxthm3.1}(b). By definition,
$$Tr_A(g,t)={\frac{1}{(1-t)^{n-1}(1-\xi t)}}.$$ From
the proof of Proposition \ref{xxprop3.3}(b), $\xi=-1$,
and the formula for $Tr_A(g,t)$ follows. The formula for
$Tr_A(g^3,t)$ follows from Lemma \ref{xxlem1.4}.

Let $\{b_1,b_2, \ldots, b_n\}$ be a basis for $A_1$ as in Theorem
\ref{xxthm3.1}(b), and let $V$ denote the subset
$\{b_3, \ldots, b_n\}$.  From the
quadratic term in the Maclaurin series expansion of $Tr_A(g,t)$ we
compute $tr(g|_{A_2}) = (n^2-3n+4)/2$. We compute $tr(g|_{A_2})$
directly by first noting that $g(b_1^2) = -b_1^2$, $g(b_2^2)  =
-b_2^2$, $g(b_1b_2) = b_1b_2$, $g(b_2b_1) = b_2b_1$,
$g(b_1b_j)=ib_1b_j$, $g(b_jb_1) = ib_jb_1$, $g(b_2b_j) = -i
b_2b_j$, $ g(b_jb_2) = -ib_jb_2$, and $g(b_\ell b_j) = b_\ell b_j$
for $\ell, j \geq 3$.  It then follows that
$$tr(g|_{A_2}) =  (-1) \epsilon + i|b_1V \cup Vb_1|
+(-i)|b_2V \cup Vb_2| + (1) m$$
where $\epsilon = |\{b_1^2,b_2^2\}|$ and $m =|\{VV \cup b_1b_2
\cup b_2b_1\}|$.  Hence $|b_1V \cup Vb_1| = |b_2V \cup Vb_2| = d
\geq n-2$, since $A$ is a domain. From the Hilbert series for the
quasi-polynomial ring $A$ we have $|A_2| = (n^2+n)/2 = \epsilon +
2d +m$, so that $m = (n^2+n)/2 - \epsilon - 2d$. Equating the two
expressions for $tr( g|_{A_2})$ gives $tr( g|_{A_2}) = (-1)
\epsilon + m = (n^2-3n+4)/2$; substituting for $m$ and solving for
$d$ gives $d = n-1 - \epsilon$, for $\epsilon = 1$ or $2$. Since
$\epsilon = 2 $ gives a contradiction, we have $\epsilon = 1$ and
$d = n-2$. It follows that $b_1V = Vb_1$, $b_2V=Vb_2$, and
$b_1^2=b_2^2$, and computing directly we have $tr(g^2|_{A_2} ) = 1
-2(n-2) + ((n^2+n)/2-2n+3) = (n^2-7n+16)/2$.

We can write $Tr_A(g^2,t)$ as
$$Tr_A(g^2,t) = \frac{1}{(1-t)^k(1-x_1t) \cdots(1-x_{n-k}t)},$$
where each $x_i \neq 1$ is a root of unity, and for each $x_i$
there is an $x_{i'}= \overline{x_i}$ [Lemma \ref{xxlem1.4}].
 By Theorem \ref{xxthm3.1}(b), $tr (g^2|_{A_1})=n-4$.
Using the Maclaurin series expansion of $Tr_A(g^2,t)$ we have
$tr(g^2| _{A_1}) = k + x_1 + \cdots + x_{n-k}$ so that $x_1 +
\cdots + x_{n-k} = (n-k)-4$ with all $x_i \neq 1$. We next
consider each of the possible solutions for the $x_i$ given by
Lemma \ref{xxlem3.2} and Lemma \ref{xxlem4.1}, and we
compare the quadratic term of each possible trace function to
$(n^2-7n+16)/2$ to show that the only possibility is the one
given in the statement of the theorem.

First we consider the cases where $(n-k)-4$ is negative: i.e. $n-k
= 0,1,2,3$. When $n-k=0$ then $g^2$ is the identity, which it is
not.  If $n-k=1$ then $$Tr(g^2,t) = \frac{1}{(1-t)^{(n-1)}(1-x_1
t)},$$  a series whose Maclaurin expansion has $t$ coefficient
$n-1+x_1 \neq n-4$.  If $n-k=2$ then
$$Tr(g^2,t) = \frac{1}{(1-t)^{(n-2)}(1-x_1 t)(1-x_2 t)},$$
and the  $t$ coefficient in the Maclaurin expansion is
$n-2 + x_1 + x_2$, which is  $n-4$ only if $x_1=x_2=-1$, giving
the $Tr(g^2,t)$ that is in the statement of the theorem.
If $n-k=3$ then the trace is
$$Tr(g^2,t) = \frac{1}{(1-t)^{(n-3)}(1-x_1 t)(1-x_2 t)(1-x_3 t)}.$$
In the Maclaurin expansion of this series the $t$ coefficient
is $n-3 + x_1 + x_2 + x_3$; if this coefficient is $n-4$, then
we have $-x_1-x_2-x_3 = 1$, in contradiction to Lemma 3.1, unless
some $x_i$ is $-1$, in which case the trace is
$$Tr(g^2,t) = \frac{1}{(1-t)^{(n-3)}(1+ t)(1-\zeta^2 t^2)}.$$
This series has Maclaurin expansion with $t^2$ coefficient
$(n^2-7n)/2 + (7 +\zeta^2)$, which is $(n^2-7n)/2 + 8$ only when
$\zeta=-1$, again the form we are trying to prove.

Next suppose that at least one of the $x_i=-1$, so without loss of
generality we assume $x_{n-k}=-1$.  Then
\[x_1 + \cdots +x_{n-k-1} = (n-k)-3\]
and by Lemma \ref{xxlem3.2} we have either (a) $n-k-3 =0$ and
$x_1=\zeta$ and $x_2=-\zeta$ for $\zeta \neq \pm 1$ a root of
unity, and the trace is
$$Tr_A(g^2,t) = \frac{1}{(1-t)^{(n-3)}(1-\zeta t)(1+\zeta t)(1+t)},$$
or (b) $n-k-3 =2$ and $2(\zeta_6 + \zeta_6^5) = 2$ and the trace
is
$$Tr_A(g^2,t) = \frac{1}{(1-t)^{(n-5)}(1-\zeta_6 t)^2
(1-\zeta_6^5 t)^2(1+t)}.$$
In the first case the coefficient of the quadratic term is
$(n^2-7n)/2+7 + \zeta^2$, and in the second case it is $(n^2-7n)/2
+ 5$. Hence we may assume
\[x_1 + x_2 + \cdots + x_{n-k} = (n-k) -4\]
and $x_i \neq \pm1$.

Next suppose that $x_j+x_\ell=0 = \zeta-\zeta$ for some $j,\ell$.
This places us again in the situation of Lemma \ref{xxlem3.2} and
we have either (a) $n-k-4 = 0$ and $x_1=\zeta'$ and $x_2=-\zeta'$,
and the trace is
$$Tr_A(g^2,t) = \frac{1}{(1-t)^{(n-4)}(1-\zeta t)(1+\zeta t)
(1-\zeta' t)(1+\zeta' t)},$$
or (b) $n-k-4=2$ and
$$Tr_A(g^2,t) = \frac{1}{(1-t)^{(n-6)}(1-\zeta_6 t)^2
(1-\zeta_6^5 t)^2(1-\zeta t)(1+\zeta t)}.$$
In the first case the coefficient of the quadratic term in the
Maclaurin expansion is $(n^2 -7n)/2+ 6 + \zeta^2 + (\zeta')^2$
(which is correct only when $\zeta$ and $\zeta'$ are $\pm 1$,
cases already considered), and in the second case it is $(n^2
-7n)/2+10 + \zeta^2$ (for which no root of unity provides the
correct value).

Next suppose that $n-k = 4$, so that
\[x_1+x_2+x_3+x_4 = 0,\]
which by multiplying by $x_1^{-1}$ reduces to a case handled by
Lemma \ref{xxlem3.2}, and the only solution is
\[\zeta - \zeta + \zeta' - \zeta' =0,\]
a case handled above.

Next we suppose that $n-k = 5$, and we are in the setting of Lemma
\ref{xxlem4.1}(3) with $x_1 + \cdots + x_{5} = 1.$ Then
$$Tr_A(g^2,t) = \frac{1}{(1-t)^{(n-5)}(1-x_1 t)
(1-x_2 t)(1- x_3 t)(1- x_4 t)(1- x_5 t)},$$
whose Maclaurin series when $x_1 + \cdots + x_{5} = 1$ begins
$1+(n-4)t +(n^2/2-7n/2 +c) t^2$ where
$$\begin{aligned}
c&= 5 + x_1(x_1 + x_2 + x_3 + x_4 + x_5) + x_2
(x_2 + x_3 + x_4 + x_5) \\
&\qquad\qquad + x_3(x_3 + x_4 + x_5) + x_4(x_4 + x_5) + x_5^2\\
&= 5 + x_1 + x_2 - x_1 x_2 +x_3 - x_1 x_3 - x_2 x_3 +
x_4^2 + x_4 x_5 + x_5^2.
\end{aligned}
$$
We will show that in all five cases $c \neq 8$ so that
$Tr_A(g^2,t)$ cannot have $n-k=5$. In case (a) $x_1 = \zeta_6,
x_2= \zeta_6^5, x_3 = \zeta, x_4 = \zeta \zeta_3, x_5 = \zeta
\zeta_3^2$ for an arbitrary $\zeta$, and we compute that $c = 5$.
In case (b) $x_1 = \zeta_{10}, x_2 = \zeta_{10}^3, x_3 =
\zeta_{10}^7, x_4 = \zeta_{15}, x_5 = \zeta_{15}^{11}$, and we
compute $c = 5 + \zeta_{10}^3$. In case (c) $x_1 = \zeta_{10}, x_2
= \zeta_{10}^3, x_3 = \zeta_{10}^9, x_4 = \zeta_{15}^8, x_5 =
\zeta_{15}^{13}$ and we compute $c = 6$. In case (d) $x_1 =
\zeta_{10}, x_2 = \zeta_{10}^7, x_3 = \zeta_{10}^9, x_4 =
\zeta_{15}^2, x_5 = \zeta_{15}^7$ and we compute $c = 5 +
\zeta_{10}$. In case (e) $x_1 = \zeta_{10}^3, x_2 = \zeta_{10}^7,
x_3 = \zeta_{10}^9, x_4 = \zeta_{15}^4, x_5 = \zeta_{15}^{14}$ and
we compute $c = 5 - \zeta_{10}^2$.

Next we suppose that $n-k = 6$ and then we are in the setting of
Lemma \ref{xxlem4.1}(4) with $x_1 + \cdots + x_{6} = 2$ and
$$Tr_A(g^2,t) = \frac{1}{(1-t)^{(n-6)}(1- \zeta_6 t)
(1- \zeta_6^5 t)(1-\zeta_{10} t)(1- \zeta_{10}^3 t)(1- \zeta_{10}^7 t)
(1-\zeta_{10}^9)},$$ whose Maclaurin series begins $1+(n-4)t
+(n^2/2-7n/2 +c) t^2$ where
$$\begin{aligned}
c &= 15-6(\zeta_6 + \zeta_6^5+ \zeta_{10}^3
+\zeta_{10}^7 + \zeta_{10}^9) + ( 1 + \zeta_6^2 + \zeta_6^4)+ \\
&\qquad\qquad 2( 1 - \zeta_{10}+
\zeta_{10}^2 -\zeta_{10}^3 + \zeta_{10}^4)+ (\zeta_{6} +
\zeta_{6}^5)(\zeta_{10} + \zeta_{10}^3 + \zeta_{10}^7 +
\zeta_{10}^9)\\
& = 15-12 + 0 + 0 + 1 = 4 \neq 8,
\end{aligned}
$$
so $n - k \neq 6$.

There is no solution if $n-k=7$ by Lemma \ref{xxlem4.1}(5),
so the last case is $n-k=8$, and
$$Tr_A(g^2,t) = \frac{1}{(1-t)^{(n-8)}(1- \zeta_6 t)^4
(1- \zeta_6^5 t)^4},$$ which has
Maclaurin series with quadratic coefficient $n^2/2-7n/2+ 2$, so
this case is also eliminated.

Hence we have shown that
$$Tr_A(g^2,t)={\frac{1}{(1-t)^{n-2}(1+t)^2}}.$$
\end{proof}

Here is a partial converse of Theorem \ref{xxthm2.4} for
mystic reflections.

\begin{proposition}
\label{xxprop4.3} Let $A$ be a quantum polynomial ring
of global dimension $n$ and let $g$ be a mystic reflection.
\begin{enumerate}
\item
There is a basis of $A_1$, say $\{b_1,b_2,\cdots,b_n\}$
such that $g(b_1)=i b_1$, $g(b_2)=-i b_2$, and $g(b_j)=b_j$
for all $j\geq 3$.
\item
$A^g$ is regular and $H_{A^g}(t)=[(1-t)^{n-2}(1-t^2)^2]^{-1}$.
\item
The subalgebra generated by $b_1$ and $b_2$ is a quantum
polynomial ring subject to one relation $b_1^2+c b_2^2=0$
for some nonzero scalar $c$. This subalgebra is also
isomorphic to $k_{-1}[x,y]$.
\item
$b_1^2$ is a normal element of $A$.
\end{enumerate}
\end{proposition}

\begin{proof} (a) Follows from Theorem \ref{xxthm3.1}(b)
and the definition of mystic reflection.

(b,c,d) For the rest of the proof, let $G$ be the group ${\mathbb
Z}/(4)$ and let $kG$ be the group algebra. Define four elements in
$kG$ as follows:
$$\begin{aligned}
f_1&={\frac{1}{4}}(1+g+g^2+g^3),\\
f_2&={\frac{1}{4}}(1-g+g^2-g^3),\\
f_3&={\frac{1}{4}}(1+i g-g^2- i g^3),\\
f_4&={\frac{1}{4}}(1-i g-g^2+ i g^3).
\end{aligned}
$$
It is well-known (and easy to check) that $\{f_1,f_2,f_3,f_4\}$
is a complete set of orthogonal idempotents of $kG$. Further
$$gf_1=f_1,\quad gf_2=-f_2,\quad gf_3=-if_3,\quad gf_4=i f_4.$$
Since $g$ has order $4$, the eigenvalues of $g$ are $1,-1,i$
and $-i$. Let
$$\begin{aligned}
A^{1}&=\{x\in A\;|\; g(x)=x\}=A^G,\\
A^{2}&=\{x\in A\;|\; g(x)=-x\},\\
A^{3}&=\{x\in A\;|\; g(x)=-i x\},\\
A^{4}&=\{x\in A\;|\; g(x)=ix\}.
\end{aligned}
$$
Then $A=A^1\oplus A^2\oplus A^3\oplus A^4$ as $A^G$-bimodules.
Viewing $f_j$ as a projection from $A$ to $f_j A$, we see that
$A^j= f_j A$ and the decomposition of $A$ corresponds
to the fact that $1=f_1+f_2+f_3+f_4$. Since each $f_j$ is a
projection, we have
$$H_{A^j}(t)=H_{f_j A}(t)=Tr_A(f_j,t).$$
Since the trace function is additive, we can compute
all $Tr_A(f_i,t)$. For example,
$$
\begin{aligned}
Tr_A(f_1,t)&={\frac{1}{4}}(\sum_{j=0}^3 Tr_A(g^j,t))\\
&={\frac{1}{4}}({\frac{1}{(1-t)^{n}}}+
               {\frac{1}{(1-t)^{n-1}(1+t)}}+\\
&\qquad\qquad\qquad\quad
               {\frac{1}{(1-t)^{n-2}(1+t)^2}}+
               {\frac{1}{(1-t)^{n-1}(1+t)}})\\
&={\frac{1}{4}}\; {\frac{(1+t)^2+2(1-t^2)+(1-t)^2}{(1-t)^{n}(1+t)^2}}\\
&={\frac{1}{4}}\; {\frac{4}{(1-t)^{n}(1+t)^2}}
={\frac{1}{(1-t)^{n}(1+t)^2}}.
\end{aligned}
$$
The second assertion of (b) follows because
$$H_{A^G}(t)=H_{A^1}(t)=Tr_A(f_1,t)=
{\frac{1}{(1-t)^{n}(1+t)^2}}.$$
Similarly we have
$$\begin{aligned}
H_{A^2}(t)=Tr_A(f_2,t)&={\frac{t^2}{(1-t)^{n}(1+t)^2}},\\
H_{A^3}(t)=Tr_A(f_3,t)&={\frac{t}{(1-t)^{n}(1+t)^2}},\\
H_{A^4}(t)=Tr_A(f_4,t)&={\frac{t}{(1-t)^{n}(1+t)^2}}.
\end{aligned}
$$
By (a), $b_1\in A^4$. Since $A$ is a domain,
$$H_{b_1A^G}(t)=H_{A^G b_1}(t)=t H_{A^G}=
t \; {\frac{1}{(1-t)^{n}(1+t)^2}}=
H_{A^4}(t).
$$
Since $b_1 A^G\subset A^4$ and $A^G b_1\subset A^4$,
we conclude that
$$A^4=b_1A^G= A^G b_1.$$
In a similar way one can show that
$$A^3=b_2 A^G= A^G b_2$$
and
$$A^2=b_1^2 A^G=A^G b_1^2=
b_2^2 A^G=A^G b_2^2.$$
Therefore $A$ is a free $A^G$-module of rank $4$
on the left and on the right. By Lemma
\ref{xxlem1.10}(a,c), $A^G$ is regular. Thus we have proved
(b).

Since both $b_1^2$ and $b_2^2$ are in $A^2$ and since
the dimension of the degree 2 part of $A^2$ is 1,
$b_1^2$ and $b_2^2$ are linearly dependent. Since
$A$ is a domain, both are nonzero. Thus there is a
nonzero scalar $c$ such that $b_1^2+cb_2^2=0$.
Changing $b_2$ by a scalar multiple, we have
$b_1^2-b_2^2=0$. By Lemma \ref{xxlem3.7}(a) the subalgebra
generated by $b_1$ and $b_2$ is a quantum polynomial
ring. Clearly $k\langle b_1,b_2\rangle/(b_1^2-b_2^2)
\cong k_{-1}[x,y]$, so we have proved (c).

Since $b_1A^G=A^G b_1$ then $b_1^2A^G=A^Gb_1^2$.
Since $b_1^2=b_2^2$, then $b_1^2$
commutes with $b_2$. Therefore $b_1^2$ is a normal
element in $A$. This is (d).
\end{proof}

\begin{example}
\label{xxex4.4} Let $B$ be the quantum algebra generated by
$b_1$ and $b_2$ subject to one relation $b_1^2-b_2^2=0$.
Let $A$ be the iterated Ore extension of $B$,
$B[b_3;\tau][b_4;\tau',\delta]$, where the automorphism
$\tau$ is determined by
$$\tau(b_1)=-b_1,\tau(b_2)=b_2,$$
the automorphism $\tau'$ is determined by
$$\tau'(b_1)=-b_1, \tau'(b_2)=b_2,\quad
\text{and}\quad \tau'(b_3)=b_3,$$
and the $\tau'$-derivation $\delta$ is
determined by
$$\delta(b_1)=\delta(b_2)=0, \quad
\text{and}\quad \delta(b_3)=b_1b_2+b_2b_1.$$
Then $A$ is a quantum polynomial ring generated by
$b_1,b_2,b_3,b_4$, subject to the following
relations
$$\begin{aligned}
b_1^2-b_2^2&=0\\
b_1b_3+b_3b_1&=0\\
b_2b_3-b_3b_2&=0\\
b_1b_4+b_4b_1&=0\\
b_2b_4-b_4b_2&=0\\
b_3b_4-b_4b_3&=b_1b_2+b_2b_1.
\end{aligned}
$$
Since any graded Ore extension of a regular
algebra is regular, $A$ is regular.
Also the Ore extension preserves the following
properties: being a domain, being noetherian,
and having Hilbert series of the form $(1-t)^{-n}$.
Thus $A$ is a quantum polynomial ring.

(a) By a direct computation, $A$ does not have a normal
element in degree 1, so there is no normal element in
$A_{\geq 1}/A_{\geq 1}^2$. But $b_1^2$ is a normal element of $A$.

(b) Let $g$ be a graded algebra automorphism of $A$ determined
by
$$g(b_1)=i b_1, g(b_2)=-i b_2, g(b_3)=b_3,g(b_4)=b_4.$$
By using a $k$-linear basis of $A$,
$$\{(b_2b_1)^s b_2^u b_3^v b_4^w \;|\; s,u,v,w\geq 0\}
\cup \{b_1(b_2b_1)^s b_2^u b_3^vb_4^w\;|\;
 s,u,v,w\geq 0\},$$
one can easily verify that
$$Tr_A(g,t)={\frac{1}{(1-t)^3(1+t)}}.$$
Therefore $g$ is a mystic reflection.

(c) The fixed subring $C:=A^g$ is generated by $b_3,b_4$ and $z:=b_1b_2$
subject to the following relations:
$$\begin{aligned}
zb_3+b_3 z&=0\\
zb_4+b_4 z&=0\\
b_3^2b_4-b_4b_3^2&=0\\
b_3b_4^2-b_4^2b_3&=0.
\end{aligned}
$$
This algebra is regular of global dimension 4. Since $z$ is normal
in $C$, then there is a normal element in $C_{\geq 1}/C_{\geq 1}^2$.
\end{example}

\begin{remark}
\label{xxrem4.5}
When $A$ is a quantum polynomial ring, we have proved that
there is only one kind
of mystic reflection: those described in this section. We expect
that, when $A$ is a noetherian regular algebra of higher
global dimension (but not a quantum polynomial ring), other
mystic reflections exist.
\end{remark}

\section{A partial Shephard-Todd-Chevalley Theorem}
\label{sec5}

In this section we prove a simple
noncommutative generalization of the Shephard-Todd-Chevalley Theorem.
The following lemma is a kind of converse of Theorem
\ref{xxthm3.1}(a).

\begin{lemma}
\label{xxlem5.1}
Let $A$ be a quantum polynomial ring with graded
algebra automorphism $g$ (not necessarily of finite order).
Suppose $g|_{A_1}$ is
a reflection of order not equal to $2$.  Then:
\begin{enumerate}
\item
There is a basis of $A$, say $\{b_1,\cdots,b_n\}$,
such that $g(b_1)=\xi b_1$ and $g(b_j)=b_j$ for
all $j>1$ and $\xi \neq -1$.
\item
$A=C[b_1;\sigma]$ where $C$ is a quantum polynomial
ring generated by $b_j$ for all $j>1$.
\item
$g$ is a quasi-reflection.
\item
$A^g$ is regular.
\end{enumerate}
\end{lemma}

\begin{proof} (a) This is clear by the definition of reflection
of $A_1$.

(b) Since all the relations of $A$ are quadratic, and the
order of $\xi$ is not $2$, $A$ becomes ${\mathbb Z}^2$-graded
after we assign $\deg b_1=(1,0)$ and $\deg b_j=(0,1)$ for
all $j\geq 2$. The assertion follows from Proposition
\ref{xxprop3.5}(c).

(c) Since $g(b_1)=\xi b_1$ and $A=C[b_1;\sigma]=\sum_{i\geq 0} b_1^i C$,
$$Tr_A(g,t)={\frac{1}{1-\xi t}} Tr_C(g,t)={\frac{1}{(1-t)^{n-1}
(1-\xi t)}}.$$
Hence $g$ is a quasi-reflection.

(d) It is clear that $A^g=C[b_1^w;\sigma^w]$ if the order
of $g$ is $w<\infty$, or $A^g=C$ if the order of $g$ is infinite.
\end{proof}

\begin{lemma}
\label{xxlem5.2}
Let $A$ be a quantum polynomial ring with $\GKdim A > 1 $and
let $G$ be a finite subgroup of $\Aut(A)$.
\begin{enumerate}
\item
If $G$ contains a quasi-reflection of order not
equal to $2$ or $4$, then $A\cong C[b;\sigma]$.
\item
Suppose that $A^G$ has finite global dimension
(and then $A^G$ is regular).
If the order of $G$ is odd, then $A\cong C[b;\sigma]$.
\item
Suppose that $A^G$ has finite global dimension
(and then $A^G$ is regular).
If $|G|=4m$ for some $m>1$ and $G$ does not contain
any reflections, then $G$ contains at least 4
mystic reflections.
\item
If $g$ is a reflection of order $2$, then
$A^g$ is regular and $A$ has a normal element
in degree $1$.
\item
If $G$ contains a reflection of order $2$, then
$A$ has a normal element in degree $1$.
\end{enumerate}
\end{lemma}

Note that a quantum polynomial ring of $\GKdim
\leq 1$ is either $k$ or $k[x]$. Both of them are
commutative and the classical Shephard-Todd-Chevalley
theorem applies.

\begin{proof}[Proof of Lemma \ref{xxlem5.2}]
(a) Let $g$ be a quasi-reflection
of order not equal to $2$ or $4$. By Theorem \ref{xxthm3.1},
$g$ is a reflection, namely, $g|_{A_1}$ is a
reflection, and the assertion follows from Lemma
\ref{xxlem5.1}.

(b) By Theorem \ref{xxthm2.4} $G$ always
contains a quasi-reflection $g$. Then the order of $g$
is not $2$ or $4$, and the assertion follows from (a).

(c) If $g$ is a mystic reflection, so is $g^3$.
So the number of mystic reflections is even.
Assume there is no reflection and that there are only
2 mystic reflections.

Let $H_{A^G}(t)=[(1-t)^{n}q(t)]^{-1}$ where $q(1) \neq 0$.
By Theorem \ref{xxthm2.5}(b),
$\deg q(t)$ is equal to the number of quasi-reflections,
which is 2. Since the roots of $q(t)$ are all roots of unity
and the coefficients of $q(t)$ are non-negative integers,
$q(1)\leq 4$. So $|G|=4$, a contradiction.

(d) Let $g$ be a reflection of $A$ of order $2$. So
there is a basis of $A_1$, say $\{b_1,\cdots,b_n\}$
such that $g(b_1)=-b_1$ and $g(b_j)=b_j$ for all $j\geq 2$.

Let $A^{+}=\{x\in A\;|\; g(x)=x\}$ and $A^{-}=\{x\in
A\;|\; g(x)=-x\}$. Then $A^{+}=A^g$ and $A=A^{+}
\oplus A^{-}$ as $A^g$-bimodules. Since $g$ is a
quasi-reflection, $Tr_A(g,t)=[(1-t)^{n-1}(1+t)]^{-1}$.
Using Molien's theorem,
$$H_{A^g}={\frac{1}{2}}({\frac{1}{(1-t)^{n}}}+
{\frac{1}{(1-t)^{n-1}(1+t)}})=
{\frac{1}{(1-t)^{n-1}(1-t^2)}},$$
and hence
$$H_{A^{-}}(t)=H_A(t)-H_{A^g}(t)
={\frac{1}{(1-t)^{n}}}-
{\frac{1}{(1-t)^{n-1}(1-t^2)}}
={\frac{t}{(1-t)^{n-1}(1-t^2)}}.$$
Since $b_1\in A^{-}$, both $b_1 A^g$ and $A^g b_1$
are subspaces of $A^{-}$. Since $A$ is a domain,
$$H_{b_1 A^g}(t)=H_{A^g b_1}(t)={\frac{t}{(1-t)^{n-1}(1-t^2)}}=
H_{A^{-}}(t).$$
This implies that $b_1 A^g=A^g b_1=A^{-}$.
Recall that $b_j\in A^g$ for all $j\geq 2$; so $b_1$ is normal.
Hence $A$ is a free module over $A^g$ on both sides.
By Lemma \ref{xxlem1.10}(a,c), $A^g$ is regular.

(e) Follows from (d).
\end{proof}

Now we are ready to prove Theorem \ref{xxthm0.5}.

\begin{theorem}
\label{xxthm5.3}
Let $A$ be a quantum polynomial ring and let $g$ be a
graded algebra automorphism of $A$ of finite order.
\begin{enumerate}
\item
If $g$ is a quasi-reflection, then the fixed subring $A^g$
is  regular.
\item
Suppose the order of $g$ is $p^m$ for some
prime $p$ and some integer $m$. If the fixed subring $A^g$ has
finite global dimension, then $g$ is a quasi-reflection.
\end{enumerate}
\end{theorem}

\begin{proof}
(a) If $g$ is a mystic reflection, the assertion follows
from Proposition \ref{xxprop4.3}(a). Now let $g$ be a
reflection. If the order of $g$ is 2, this follows
from Lemma \ref{xxlem5.2}(d). If the order of $g$ is larger
than $2$, the assertion follows from Lemma \ref{xxlem5.1}(d).

(b) Suppose $A^g$ is regular. We use induction
on $m$. First assume $m$=1. By Theorem \ref{xxthm2.4},
$G:=\langle g\rangle$ contains a quasi-reflection $g^i$
(and hence a reflection since $p\neq 4$). Since
$p$ is prime, $g$ is a power of $g^i$. By Lemma
\ref{xxlem1.4}, $g$ is a quasi-reflection if and only
if $g^i$ is. So we are done.

Now we assume the order of $g$ is $p^m$ for $m\geq 2$.
By Theorem \ref{xxthm2.4}, $G$ contains a
quasi-reflection, say $g^i$ for some $i$.
If $g^i$ is a mystic
reflection, then the order of $g^i$ is $4$, and hence
$p=2$. There are at most two elements in $G$ of order 4.
By Lemma \ref{xxlem5.2}(c), the order of $G$ is 4.
Hence $i=1$ or $3$, and $g$ is a mystic reflection, and
hence a quasi-reflection, completing the argument.

As the above paragraph showed, there are at most 2
mystic reflections in $G$ since there are at most
two elements of order $4$. Similarly, there is
at most one element of order $2$ in $G$. Further,
if $G$ contains a mystic reflection, then
the element of order 2 is not a quasi-reflection
by Lemma \ref{xxlem4.2}. If $G$ contains
only one quasi-reflection $g$, then $g=g^{-1}$
since $g^{-1}$ is also a quasi-reflection by Lemma
\ref{xxlem1.4}. Thus $|G|=2$ and this
case has been taken care of when $m=1$.

Now suppose that we are not in the cases discussed
in the above two paragraphs; then $G$ contains a
reflection $h$ of order not equal to $2$. Without loss
of generality we may write this element as
$h:=g^{p^w}$ for some $w<m$. So the order of
$h$ is $p^{m-w}$. If $w=0$, then we are
done. Hence we assume that $w>0$.
Let $\{b_1,\cdots,b_n\}$ be a basis of $A_1$
such that $g(b_j)=\xi_j b_j$ for all $j$; further let
$h(b_1)= \xi b_1$ and $h(b_j)=b_j$ for all
$j>2$. Clearly, $\xi=\xi_1^{p^w}$. Since the
order of $\xi$ is equal to the order of $h$,
which is $p^{m-w}$, the order of $\xi_1$ is
$p^m$. By Lemma \ref{xxlem5.1}(a,b),
$A=C[b_1;\sigma]$. Let $A'=A^{h}=
C[b_1^{p^{m-w}};\sigma^{p^{m-w}}]$. Then
the $G$ action on $A$ induces a $G':=G/(h)$
action on $A'$. Reassigning the degree as $\deg
b_1^{p^{m-w}}=1$, $A'$ is a quantum polynomial
ring. It is clear that $A^G= A'^{G'}$.
Since $w>0$, $G'$ is generated by $g':=g(h)\in G'$
of order $p^{w}$ which is less than $p^m$.
Since $A'^{G'}$ is regular, by induction $g'$
is a quasi-reflection. Finally we have two cases
to deal with. First we assume that
$g'$ is not a reflection of $A'$. Then
it is a mystic reflection of $A'$. So $p^{w}=2^2$.
By the choice of $\{b_j\}$ we have
$$g'(b_1^{2^{m-2}})=g(b_1^{2^{m-2}})=
\xi_1^{2^{m-2}}b_1^{2^{m-2}}=\pm i b_1^{2^{m-2}}.$$
Without loss of generality we only consider the $+ i$
case since the $- i$ case is similar.
Up to a permutation we have
$g'(b_2)=g(b_2)=- i b_2$ and $g'(b_j)=g(b_j)=b_j$
for all $j\geq 3$. By Proposition \ref{xxprop4.3}(c),
$(b_1^{2^{m-2}})^2=c b_2^2$ for some nonzero
scalar $c$, but this is impossible in $A$.
This leaves us the second and the last case:
$g'$ is a reflection of $A'$. By the
choice of $\{b_j\}$, we have
$$g'(b_1^{p^{m-w}})=g(b_1^{p^{m-w}})=
\xi_1^{p^{m-w}}b_1^{p^{m-w}}\neq b_1^{p^{m-w}}$$
and $g'(b_j)=g(b_j)=\xi_jb_j$ for all $j\geq 2$.
By the definition of reflection, we conclude
that $\xi_j=1$ for all $j\geq 2$. Therefore $g$
is a reflection.
\end{proof}

Finally we give an example showing that a reflection
of order $2$ does exist for some $A$ not isomorphic to
$C[b_1;\sigma]$.

\begin{example}
\label{xxex5.4}
Let $A$ be the Rees ring of the first Weyl algebra
with respect to the standard filtration.
So $A$ is generated by $x,y$ and $z$ subject to the
relations
$$xy-yx=z^2,\quad \text{$z$ is central}.$$
Let $g$ be the automorphism of $A$ determined by
$$g(x)=x,g(y)=y,\quad \text{and}\quad g(z)=-z.$$
Then $g$ is of order 2. Since $z$ is central, it
is easy to check that $Tr_A(g,t)=[(1-t)^2(1+t)]^{-1}$.
Hence $g$ is a quasi-reflection and $g|_{A_1}$ is
a reflection. So $g$ is a reflection in the sense of
Definition \ref{xxdefn3.9}.

(a) $A^g$ is regular by Theorem \ref{xxthm5.3}(a).

(b) We claim that $A\not\cong C[b;\sigma]$. Suppose
$A=C[b;\sigma]$, then it is easy to check that $z$ (up to a
scalar) is  the only normal element in degree 1.
Thus $b=z$ and $C=A/(z)$ is commutative. Since
$b=z$ is central, then $A$ is commutative, a contradiction.
Thus $A\not\cong C[b;\sigma]$.

(c) The regular fixed subring $A^g$ is generated by $x$ and $y$, and is
isomorphic to $U(L)$, where $L$ is the Lie algebra
$kx+ky+kw$ where $w=z^2=[x,y]$. Hence
the fixed subring $A^g$ is a regular ring that is different than
$A$.  We note that $U(L)$ is a two-generated regular ring of
dimension 3, hence Proposition \ref{xxprop6.4} will show that
it does not have any quasi-reflections of finite order, so it
is rigid.  Hence $U(L)$ can be a fixed subring
of a regular ring, but it cannot be the fixed subring of a
finite group acting on itself.

We will examine the Rees ring of $A_n(k)$ in the next
section [Proposition \ref{xxprop6.7} and Corollary \ref{xxcor6.8}].
\end{example}

\section{Rigidity theorems}
\label{sec6}

In this section we prove the rigidity theorems \ref{xxthm0.1}
and \ref{xxthm0.2} stated in the introduction.

\begin{lemma}
\label{xxlem6.1}
Let $A$ be a noetherian regular algebra. Suppose $A$ has
no quasi-reflection of finite order.  Then:
\begin{enumerate}
\item
For every finite group $G\subset \Aut(A)$, $A^G$ has infinite
global dimension.
\item
For every finite group $G\subset \Aut(A)$, $A^G$ is not
isomorphic to $A$.
\end{enumerate}
\end{lemma}

\begin{proof} (a) This is Theorem \ref{xxthm2.4}.

(b) If $A^G$ is isomorphic to $A$, then $A^G$ has finite
global dimension, and so the assertion follows from (a).
\end{proof}

\begin{theorem}
\label{xxthm6.2} Let $A$ be a quantum polynomial ring.
Suppose that one of the following condition holds.
\begin{enumerate}
\item
$A$ has no element $b$ of degree 1 such that $b^2$ is normal
in $A$.
\item
$A$ has no normal element in degree 1, and no
subalgebra isomorphic to $k_{-1}[b_1,b_2]$.
\end{enumerate}
Then the following conditions hold.
\begin{enumerate}
\item[(i)]
$A$ has no quasi-reflection of finite order.
\item[(ii)]
For every finite group $G\subset \Aut(A)$, $A^G$ has infinite
global dimension.
\item[(iii)]
For every finite group $G\subset \Aut(A)$, $A^G$ is not
isomorphic to $A$.
\end{enumerate}
\end{theorem}

\begin{proof} By Lemma \ref{xxlem6.1} we only need
to show (i). So we consider the two cases.

(a) If $A$ has a reflection of finite order, then by Lemmas
\ref{xxlem5.1}(b) and \ref{xxlem5.2}(e), $A$ has a normal
element $b$ in degree 1. Then $b^2$ is normal, a contradiction.
If $A$ has a mystic reflection, by Proposition \ref{xxprop4.3}(d),
$b_1^2$ is normal, a contradiction. So the assertion (i)
follows.

(b) As in case (a), if $A$ has a reflection of finite order,
$A$ has a normal element $b$ in degree 1. This is a
contradiction. If $A$ has a mystic reflection, $A$
has a subalgebra isomorphic to $k_{-1}[b_1,b_2]$
by Proposition \ref{xxprop4.3}(c).
\end{proof}

\begin{corollary}
\label{xxcor6.3}
Let $S$ be a non-PI Sklyanin algebra of global dimension $n
\geq 3$. Then $S$ has no quasi-reflection of finite order.
As a consequence, $S^G$ is not regular, and so $S$ is not
isomorphic to $S^G$, for any
non-trivial finite group $G$ of graded algebra automorphisms.
\end{corollary}

\begin{proof} By Theorem \ref{xxthm6.2} it suffices to check
that $S$ has no element $b$ in degree 1 such that $b^2$ is normal.

Associated to $S$ there is a triple $(E,\sigma,{\mathcal L})$
where $E\subset {\mathbb P}^{n-1}$ is an elliptic curve of
degree $n$, ${\mathcal L}$ is an invertible line
bundle over $E$ of degree $n$ and $\sigma$ is an automorphism
of $E$ induced by the translation. The basic properties of $S$
can be found in \cite{ATV1} for $n=3$, \cite{SmSt} for $n=4$, and
\cite{TV} for $n\geq 5$. Associated to $(E,\sigma,{\mathcal L})$
one can construct the twisted homogeneous coordinate ring,
denoted by $B(E,\sigma,{\mathcal L})$. Then there is a
canonical surjection
$$\phi:S\to B(E,\sigma,{\mathcal L})=:B$$
such that
$\phi$ becomes an isomorphism when restricted to degree 1
piece. This statement was proved by Tate-Van den Bergh
\cite[(4.3)]{TV} for $n\geq 5$, by Smith-Stafford
\cite[Lemma 3.3]{SmSt} for $n=4$ and by Artin-Tate-Van den Bergh
\cite[Section 6]{ATV1} for $n=3$. If $S$ is non-PI, then
$\sigma$ has infinite order. Hence $B$ is so-called projectively
simple \cite{RRZ}, which means that any proper factor ring
of $B$ is finite dimensional. Also note that the
GK-dimension of $B$ is 2.

Suppose that there is a $b\in S$ of degree 1, such that $b^2$ is
normal. Let $\bar{b}=\phi(b)\in B$. Since $\phi$ is an
isomorphism in degree 1, $\bar{b}\neq 0$. Now
a basic property of $B$ is that it is a domain.
Hence $\bar{b}^2\neq 0$, and since $b^2$ is normal, so is
$\bar{b}^2$. Therefore $B/(\bar{b}^2)$ is an infinite
proper factor ring of $B$, which contradicts the fact
that $B$ is projectively simple.
\end{proof}

We note that an extensive calculation shows that Corollary \ref{xxcor6.3}
is also true for 3 dimensional PI Sklyanin algebras, suggesting
that the PI hypothesis may not be necessary.

Next we give a class of regular rigid algebras that are
not quasi-polynomial rings.

\begin{proposition}
\label{xxprop6.4}
Let $A$ be a noetherian regular algebra of global
dimension 3 that is generated by two elements in degree 1.
Then $A$ has no quasi-reflection of finite order, and hence
no regular fixed subrings $A^G$ for $G$ a finite group.
\end{proposition}

\begin{proof} By the Artin-Schelter classification \cite{ASc},
the Hilbert series of $A$ is
$$H_A(t)={\frac{1}{(1-t)^2(1-t^2)}}.$$
In particular, $A$ has GK-dimension 3 and has two relations
of degree 3. Let $g$ be a possible quasi-reflection of $A$
of finite order. Then the trace of $g$ is
$$Tr_A(g,t)={\frac{1}{(1-t)^2(1-\xi_1 t)(1-\xi_2 t)}}$$
where $\xi_1$ and $\xi_2$ are roots of unity by Lemma
\ref{xxlem1.6}(d).

Let $\{b_1,b_2\}$ be a basis of $A$ such that $g(b_i)=x_i b_i$
for $i=1,2$, where $x_1$ and $x_2$ are root of unity.
Comparing the coefficients of $t$ in the Maclaurin series expansion of
$Tr_A(g,t)$, we obtain that
$$tr (g|_{A_1})=x_1+x_2=1+1+\xi_1+\xi_2.$$
By Lemma \ref{xxlem3.2} there are three solutions:

Solution 1: $\xi_1=\xi_2=-1$, $x_1=-x_2$.

Solution 2: $\xi_1=-1$, $x_1=1$, $\xi_2=x_2$ up to
a permutation.

Solution 3: $\{x_1,x_2,-\xi_1,-\xi_2\}=
\{\zeta_6,\zeta_6,\zeta_6^5,\zeta_6^5\}$
up to a permutation.

Next we show that each of these is impossible.

Solution 1: Since $\xi_1=\xi_2=-1$, $tr (g|_{A_2})=2$.
The eigenvalues of $g|_{A_2}$ are $x_1^2$ with eigenspace
$k b_1^2+kb_2^2$ and $-x_1^2$ with eigenspace $kb_1b_2+kb_2b_1$.
So $tr (g|_{A_2})=0$, a contradiction.

Solution 2: Since $\xi_1=-1$ and $\xi_2=x_2$,
$tr (g|_{A_2})=2+x_2+x_2^2$. Applying $g$ to
the space $A_2$, we see that $tr (g|_{A_2})=
1+2x_2+x_2^2$. Hence $x_2=1$. This is impossible
since $g$ is not the identity.

Solution 3: If $x_1=x_2$, then $Tr_A(g,t)=H_A(x_1t)$
which shows that $g$ is not a quasi-reflection.
Hence $x_1\neq x_2$. Up to a permutation
we may assume $x_1=-\xi_1=\zeta_6$ and $x_2=-\xi_2
=\zeta_6^5$. Expanding $Tr_A(g,t)$, we have
$$Tr_A(g,t)={\frac{1}{(1-t)^2(1+\zeta_6 t)(1+\zeta_6^5 t)}}
=1+t+t^2+2t^3+\cdots.$$
Consequently, $tr (g|_{A_3})=2$. Now consider $g|_{A_3}$.
The eigenvalues of $g|_{A_3}$ are either $-1
(=\zeta_6^3=(\zeta_6^5)^3)$, $\zeta_6$ or $\zeta_6^5$.
So we have
$$2=tr (g|_{A_3})=n_1(-1)+n_2\zeta_6+n_3\zeta_6^5,
\quad n_1,n_2,n_3\geq 0$$
where $n_1+n_2+n_3= 6$ is the $\dim A_3$. But this
is impossible.
\end{proof}

Proposition \ref{xxprop6.4} applies to a noetherian
graded down-up algebra $A= A(\alpha,\beta,0)$, where
$\beta \neq 0$ (see \cite{BR, KMP}). This
algebra is generated by $d,u$
subject to the two relations:
$$du^2=\alpha udu + \beta u^2 d \;\; \text{   and   }
\;\; d^2u=\alpha dud + \beta ud^2.$$
It is a noetherian regular algebra of global
dimension 3, and so by the above proposition, $A$ has no
quasi-reflection of finite order.

Let ${\mathfrak g}$ be a Lie algebra finite dimensional
over $k$ with Lie bracket $[\;,\;]$. Let $\{b_1,\cdots,b_n\}$
be a $k$-linear basis of ${\mathfrak g}$. The homogenization
of $U({\mathfrak g})$, denoted by $H({\mathfrak g})$, is
defined to be its Rees ring with respect
to the standard filtration of $U({\mathfrak g})$. It is
a connected graded
algebra generated by the vector space ${\mathfrak g}+kz$
subject to the relations
$$b_iz=zb_i\quad \text{and}\quad b_ib_j-b_jb_i=[b_i,b_j]z$$
for all $i,j$. To distinguish it from the Lie product, we
use  $\lfloor x,y\rfloor$ to denote $xy-yx$ in an algebra.
Then the relations of $H({\mathfrak g})$ can be written as
$$\lfloor b_i,z\rfloor =0, \quad \text{and}\quad
\lfloor b_i,b_j\rfloor =[b_i,b_j]z.$$
It is well-known that $H({\mathfrak g})$ is a quantum
polynomial ring of dimension $n+1$ \cite[\S 12]{Sm2}.
By definition, $z$ is a central element such that
$H({\mathfrak g})/(z-1) \cong U({\mathfrak g})$ and that
$H({\mathfrak g})/(z) \cong k[{\mathfrak g}]$.

\begin{lemma}
\label{xxlem6.5}
Let ${\mathfrak g}$ be a finite dimensional Lie algebra
with no 1-dimensional Lie ideal, and let $H=H({\mathfrak g})$.
Then:
\begin{enumerate}
\item
Up to a scalar, $z$ is the only nonzero normal element of $H$ in
degree 1.
\item
Up to a scalar, $z$ is the only normal element in $H-k$ such that
$H/(z)\cong k[{\mathfrak g}]$.
\item
$H\not\cong C[b;\sigma]$ as graded rings.
\item
$H$ does not have any quasi-reflection of finite order.
\item
Suppose that ${\mathfrak g}'$ is another Lie algebra with no
1-dimensional Lie ideal. If $H\cong H({\mathfrak g}')$ as
ungraded algebras, then ${\mathfrak g}\cong {\mathfrak g}'$
as Lie algebras.
\end{enumerate}
\end{lemma}

\begin{proof} (a) Suppose there is another normal
element in degree 1. We may write it as $b+\xi z$
for some $0\neq b\in {\mathfrak g}$ and some
$\xi\in k$. Since $b+\xi z$ is normal, for every
$0\neq x\in {\mathfrak g}$, there are elements $y\in
{\mathfrak g}$ and $\xi'\in k$ such that
\begin{equation}
\label{6.5.1}
x(b+\xi z)=(b+\xi z)(y+\xi' z).
\tag{6.5.1}
\end{equation}
Modulo $z$ we have $xb=by$ in $k[{\mathfrak g}]$, and hence
$y=x$. Thus \eqref{6.5.1} implies that
$$(b+\xi z)\xi' z=\lfloor x,b+\xi z\rfloor =\lfloor x,b\rfloor =[x,b]z.$$
This implies that $[x,b]=\xi' b$. Since $x$ is
arbitrary, $kb$ is a 1-dimensional Lie ideal.
This yields a contradiction.

(b) Let $w\in H-k$ be another normal element in $H$ such that
$H/(w)\cong k[{\mathfrak g}]$. Then $\lfloor H,H\rfloor \subset
wH=Hw$. Consequently,
$$[{\mathfrak g},{\mathfrak g}]z=\lfloor {\mathfrak g},
{\mathfrak g}\rfloor \subset wH.$$
Since ${\mathfrak g}$ has no 1-dimensional Lie
ideal, the Lie ideal $[{\mathfrak g},{\mathfrak g}]$
must have dimension at least $2$. Pick two linearly independent
elements $b_1,b_2\in [{\mathfrak g},{\mathfrak g}]$, we have
$b_1 z, b_2 z\in wH$. Write $b_1z=c_1 w$ and $b_2 z=c_2 w$.
Since $H$ is a domain, $\deg c_i+\deg w=\deg b_1 z=2$.
Since $b_1$ and $b_2$ are linearly independent, the
degree of $w$ cannot be $2$. Hence $\deg w=1$. A simple
calculation shows that $w=z$ up to a scalar.

(c) By (a), $z$ is the only normal element in degree 1.
If $H\cong C[b;\sigma]$, then $b$ must be $z$ and $\sigma
=Id_C$. In this
case $C=H/(b)=H/(z)$, which is isomorphic to the commutative
polynomial ring. Therefore $H\cong C[b;\sigma]$ is commutative,
a contradiction.

(d) Suppose $g$ is a quasi-reflection. If $g$ is a
reflection of order larger than $2$, by Lemma
\ref{xxlem5.1}(b), $H\cong C[b;\sigma]$. This is
impossible by (c).

If $g$ is a reflection of order $2$, by the proof of
Lemma \ref{xxlem5.2}(d), there is a basis of $H_1$,
$\{b,c_1,\cdots,c_n\}$, so that $b$ is a normal element
of $H$ and $g(b)=-b$ and $g(c_i)=c_i$ for all $i$. By (a),
$z$ is the only normal element in degree 1.
Hence
$$b=z,\quad c_i=b_i+\xi_i z$$
for a basis $\{b_i\}$ of ${\mathfrak g}$ and for
some $\xi_i\in k$. Now we compute $g(\lfloor c_i,c_j\rfloor )$
in two ways:
$$g(\lfloor c_i,c_j\rfloor )=\lfloor g(c_i),g(c_j)\rfloor =
\lfloor c_i,c_j\rfloor =\lfloor b_i+\xi_i z,b_j+\xi_j z\rfloor =
\lfloor b_i,b_j\rfloor =[b_i,b_j]z$$
and
$$g(\lfloor c_i,c_j\rfloor )=g([b_i,b_j]z)=g([b_i,b_j])g(z)
=([b_i,b_j]+\xi z)(-z)$$
for some $\xi\in k$. The only possible solution is $\xi=0$ and
$[b_i,b_j]=0$. But we can choose $i,j$
such that $[b_i,b_j]\neq 0$, which yields a
contradiction.

Finally if $g$ is a mystic reflection (of order $4$),
there there are two linearly independent
elements $c_1$ and $c_2$ in $H_1$ such that
$c_1^2=c_2^2$ [Proposition \ref{xxprop4.3}(c)].
Since $H/(z)$ is a commutative polynomial ring,
$c_1=\pm c_2$ in $H/(z)$. Up to a scalar, we may
assume $c_1=b+z$ and $c_2=b+\tau z$ where $b$ is
a nonzero element in ${\mathfrak g}$ and
$1\neq \tau\in k$. In this form, one can easily check
that $c_1^2\neq c_2^2$ in $H$. Therefore $H$ has
no mystic reflection.

(e) Let $H'=H({\mathfrak g}')$. Let $f: H\to H'$ be
an isomorphism of (ungraded) algebras. By (b),
$f(z)=\xi z$ for some nonzero scalar $\xi$. There
is an automorphism of the graded algebra $H'$ sending
$\xi z$ to $z$. So we can assume that $f(z)=z$.

Let $\{b_1,\cdots,b_n\}$ be a basis of ${\mathfrak g}$.
For every $j$, write $f(b_j)=\xi_j+\sigma(b_j)$ where
$\xi_j\in k$ and $\sigma(b_j)\in H'_{\geq 1}$. We claim
that $z\mapsto z:=\sigma(z), b_j\mapsto \sigma(b_j)$
defines an isomorphism from $H$ to $H'$. First we show
that $\sigma$ defines an algebra homomorphism, namely,
$\sigma$ preserves the defining relations. Recall that
the defining relations of $H$ are
$$\lfloor b_j,z\rfloor =0 \quad \text{and}\quad
\lfloor b_j,b_f\rfloor=[b_j,b_f] z.$$
Since $\sigma(z)=z$ is central in $H'$, we have
$\lfloor \sigma(b_j),z\rfloor =0$, namely, $\sigma$
preserves the first set of relations. Applying $f$
to the second set of relations, we have
$$\lfloor f(b_j),f(b_f)\rfloor =f([b_j,b_f])f(z)=
f([b_j,b_f])z.$$
Since $\lfloor H',H'\rfloor\subset z{\mathfrak g}' H'$,
$f([b_j,b_f])\in {\mathfrak g}' H'$. Hence
$\sigma([b_j,b_f])=f([b_j,b_f])$ after extending
$\sigma$ linearly. Now
$$\lfloor \sigma(b_j),\sigma(b_f)\rfloor
=\lfloor f(b_j)-\xi_j,f(b_f)-\xi_f\rfloor
\qquad\qquad\qquad\qquad$$
$$\qquad\qquad\qquad\qquad
= \lfloor f(b_j),f(b_f)\rfloor
=f([b_j,b_f])z=\sigma([b_j,b_f])\sigma(z).
$$
Therefore $\sigma$ preserves the second set of the defining
relations. Thus we have proved that $\sigma$ is an algebra
homomorphism. Since $\{b_1,\cdots,b_n,z\}$ generates $H$ and
$f$ is an isomorphism, then $\{f(b_1),\cdots, f(b_n),z\}$
generates $H'$. Hence $\{\sigma(b_1),\cdots, \sigma(b_n),z\}$
generates $H'$ also, and we have shown that $\sigma$
is an algebra isomorphism from $H$ to $H'$.

Note that $\sigma(H_{\geq 1})\subset H'_{\geq 1}$.
Since $\sigma$ is an isomorphism, $\sigma(H_{\geq 1})=
H'_{\geq 1}$. Since $H$ is generated in degree 1,
it has a natural filtration
$$\{F_{-j}H:=(H_{\geq 1})^j=H_{\geq j}\;|\; j
\in {\mathbb Z}\}.$$
The same is true for $H'$. Thus $\sigma$ is a
filtered isomorphism that induces a graded algebra
isomorphism $\tau:=\gr \sigma: \gr H\to \gr H'$. Since
$\gr H=H$, $\tau$ is a graded isomorphism from
$H$ to $H'$ sending $z$ to $z$.

For every $b\in {\mathfrak g}$ write
$\tau(b)=\phi(b)+\chi(b) z$ where $\phi(b)\in
{\mathfrak g}'$ and $\chi$ is a linear map
from ${\mathfrak g}$ to $k$. We claim that
$\phi:{\mathfrak g}\to {\mathfrak g}'$ is a
Lie algebra isomorphism. Since $\tau(z)=z$,
$\phi$ is an isomorphism of $k$-vector spaces.
To show $\phi$ preserves the Lie product, we use
the following direct computation:
$$
\begin{aligned}
\phi([b_j,b_f])z &= \tau([b_j,b_f])z-\chi([b_j,b_f])z^2=
\tau([b_j,b_f]z)-\xi z^2\\
&=\tau(\lfloor b_j,b_f\rfloor )-\xi z^2=
\lfloor \tau(b_j),\tau(b_f)\rfloor-\xi z^2\\
&=\lfloor \phi(b_j)+\chi(b_j)z,\phi(b_f)+\chi(b_f)z\rfloor-\xi z^2\\
&=\lfloor \phi(b_j),\phi(b_f)\rfloor-\xi z^2
=[\phi(b_j),\phi(b_f)]z-\xi z^2.\\
\end{aligned}
$$
Thus $\phi([b_j,b_f])=[\phi(b_j),\phi(b_f)]$ and
$\xi:=\chi([b_j,b_f])=0$. Therefore $\phi$ is
a Lie algebra isomorphism from ${\mathfrak g}$
to ${\mathfrak g}'$.
\end{proof}

\begin{proof}[Proof of Theorem \ref{xxthm0.2} (a)]
By Lemma \ref{xxlem6.5}(d), $H:=H({\mathfrak g})$
does not have any quasi-reflections of
finite order. By Theorem \ref{xxthm2.4},
for any finite group $G\subset \Aut(H)$,
$H^G$ does not have finite global
dimension. Thus $H^G\cong H({\mathfrak g})$
implies that $G$ is trivial. Since $G=\{1\}$ then
$H({\mathfrak g})\cong H({\mathfrak g}')$, which
implies ${\mathfrak g}\cong {\mathfrak g}'$
by Lemma \ref{xxlem6.5}(e).
\end{proof}

\begin{example}
\label{xxex6.6}
This example shows that the condition about
non-existence of 1-dimensional Lie ideals
in Theorem \ref{xxthm0.2} is necessary.

Let ${\mathfrak g}$ be the 2-dimensional
solvable Lie algebra $kx+ky$ with $[x,y]=y$.
Then $ky$ is a 1-dimensional Lie ideal. The
homogenization $H({\mathfrak g})$ of $U({\mathfrak g})$
is generated by $x,y,z$ subject to the following
relations
$$xy-yx=yz, \quad zx=xz, zy=yz.$$
It is easy to see that $H({\mathfrak g})$ is isomorphic
to an Ore extension $k[x,z][y;\sigma]$ where $\sigma(x)=x+z$
and $\sigma(z)=z$.
Let $g$ be an automorphism of $H({\mathfrak g})$ determined
by
$$g(x)=x, g(z)=z, \quad \text{and}\quad g(y)=-y.$$
It is easy to see that $g$ is a reflection of
$H({\mathfrak g})$.
The fixed subring of $H({\mathfrak g})$ is
isomorphic to $k[x,z][y^2;\sigma^2]$.
There is an isomorphism $\phi:H({\mathfrak g})
\to  H({\mathfrak g})^g$ defined by
$$\phi: x\mapsto x,\quad y\mapsto y^2,\quad
z\mapsto 2z.$$
\end{example}

Finally let us consider the proof of Theorem \ref{xxthm0.2}(b).
Let $A$ be the Rees ring of the Weyl algebra $A_n(k)$ with
respect to the standard filtration; then $A$ is
the algebra with generating set $\{x_i,y_i,z: i=1, 2, \ldots ,
n\}$ subject to the relations $x_iy_i-y_ix_i=z^2$ for $i=1,2,
\ldots , n$, and with all other generators commuting.  The algebra $A$
is a regular domain of dimension $2n+1$ \cite[3.6]{Le}
with Hilbert series $H_A(t) = 1/{(1-t)^{2n+1}}.$  We first find
the reflection groups of $A$.

\begin{proposition}
\label{xxprop6.7}
Let $A$ be the Rees ring of the Weyl algebra $A_n(k)$.
\begin{enumerate}
\item[(a)] If $g$ is a quasi-reflection of $A$, then
$g$ is a reflection of the form $g(x_i) = x_i + a_iz,
g(y_i) = y_i+b_iz, \text{ and } g(z) = -z$
for elements $a_i, b_i \in k$.
\item[(b)] If $G$ is a finite group of graded
automorphisms of $A$ such that $A^G$ is regular, then $G =
\{Id, g\}$ for a reflection $g$.
\end{enumerate}
\end{proposition}
\begin{proof} (a) Let $g$ be a quasi-reflection of $A$.
Then $$ Tr(g,t) = \frac{1}{(1-t)^{2n}(1-\xi t)}$$ for some scalar
$\xi$. Since $z$ is the only central element of degree $1$, we
must have that $g(z) = \lambda z$ for some scalar $\lambda$.

Suppose that $\lambda \neq 1$.  Since $\langle z \rangle$ is
$g$-invariant, $g$ induces an automorphism $\bar{g}$ of $\bar{A} =
A/\langle z \rangle.$ Since $\ds Tr_{\bar{A}}(\bar{g},t) =
(1-\lambda t)Tr_A(g,t) = (e_{\bar{g}}(t))^{-1}$, we have that
$Tr_{\bar{A}}(\bar{g},t) = {(1-t)^{-2n}}$, and
$\bar{g}$ must be the identity on $\bar{A} = k[x_i, y_i, i=1, 2,
\ldots , n]$.  Then $g(x_i) = x_i + a_iz, g(y_i) = y_i+b_iz,
\text{ and } g(z) =  \lambda z$.
In order that the relations of $A$ are preserved by $g$ it
follows that $\lambda^2=1$, so that $g$ has the form stated.

Now suppose that $g(z) =z$.  If $g$ is a reflection, then by
Lemma \ref{xxlem5.1} (b) and the proof of Lemma \ref{xxlem5.2}(d)
there is a basis $\{b_1, b_2, \ldots ,
b_{2n+1}\}$ of $A_1$ such that $g(b_1) = \xi b_1$, $g(b_i) =
b_i$ for $i \geq 2$, and $b_1$ is a normal element of $A$.  Since
there are no normal elements in $A_1$ other than multiples of $z$,
this cannot be, and hence there are no reflections with $g(z)=z$.

Now suppose that $g$ is a mystic reflection with $g(z) = z$. Then
by Proposition 4.2 there is a basis $\{b_1, b_2, \ldots ,
b_{2n+1}\}$ of $A_1$ such that $g(b_1) = i b_1, g(b_2) = -i b_2,$
$g(b_i) = b_i$ for $i \geq 3$, and $b_1^2$ is a normal
element of $A$.  Since multiples of $z$ are the only elements of $A_1$
which square to normal elements, we have shown that there are no mystic
reflections of $A$.  Hence (a) follows.

(b)  Suppose that $G$ is a finite group of graded automorphisms of
$A$ such that $A^G$ is regular.  Then $G$ must contain a
quasi-reflection $g_1$ by Theorem \ref{xxthm2.4}.  Suppose that $G$ contains
another quasi-reflection $g_2 \neq g_1$.  By (a) these quasi-reflections
are reflections that can be represented on $A_1$ by
matrices
\[ M_{g_1} = \left[\begin{array}{cc} I & \bar{0} \\
\bar{v} & -1 \end{array}\right],
M_{g_2} = \left[\begin{array}{cc} I & \bar{0} \\
\bar{u} & -1
\end{array}\right] \]
where $I$ is a $2n \times 2n$ identity matrix and
$\bar{u} \neq \bar{v}$.  Then $g_1 g_2$ is represented
by the product matrix \[ \left[\begin{array}{cc} I & \bar{0} \\
\bar{v} - \bar{u} & +1 \end{array}\right],\]
which has infinite
order.  Hence $G$ can contain exactly one quasi-reflection. Since
$A^G$ is regular, its Hilbert series has the form $$ H_{A^G}(t)
= \frac{1}{(1-t)^{2n+1}q(t)}$$ where $q(t)$ is a product of
cyclotomic polynomials. By Theorem \ref{xxthm2.5} (b), $\deg q(t)$ is the
number of quasi-reflections in $G$, and hence must be $1$.
Consequently $q(t) = 1+t$.  Also by Theorem \ref{xxthm2.4}(b)
$q(1) = 2 =|G|$.  Thus $G = \{Id,g\}$ for a reflection $g$.

Note that $A^g$ is regular by Theorem \ref{xxthm5.3} (a).
\end{proof}

\begin{corollary}
\label{xxcor6.8}
Let $A$ be the Rees ring of the Weyl algebra $A_n(k)$.  Then $A$ is
not isomorphic (as an ungraded algebra) to $A^G$  for any
finite group of graded automorphisms.
\end{corollary}

\begin{proof} If $A^g$ has infinite global dimension,
then $A^g\not\cong A$.

If $A^g$ has finite global dimension, by Proposition
\ref{xxprop6.7}(b), $G=\{Id,g\}$ for a reflection $g$
whose matrix on $A_1$ is of the form
$$ \left[\begin{array}{cc} I & \bar{0} \\ \bar{v} &
-1\end{array}\right]$$
for $\bar{v} = [a_1,b_1, \ldots , a_n, b_n]$.  A computation
shows that if $X_i = x_i + \frac{a_i}{2}z$
and $Y_i = y_i + \frac{b_i}{2}z$ for $i=1,2, \ldots, n$, then
$A^g$ is generated by the set
$$\{X_1,Y_1, X_2, Y_2, \ldots,X_n,Y_n, z^2\}$$
subject to the relations $X_iY_i-Y_iX_i=z^2$, with all other
generators commuting. In particular, $A^g$ is generated by $2n$
elements since $z^2=X_iY_i-Y_iX_i$. But $A$ is (minimally)
generated by $2n+1$ elements. Therefore $A^g\not\cong A$.
\end{proof}

\section{Further questions}
\label{xxsec7}

The results we have obtained suggest that the invariant theory of
Artin-Schelter regular rings merits further study.  We conclude by
describing a few directions that seem particularly interesting.

\subsection*{Bi-reflections}
In the case that $A=k[x_1, \cdots, x_n]$ Kac and Watanabe
\cite{KW} and Gordeev \cite{G} independently proved that if $A^G$
is a complete intersection and $G$ is a finite subgroup of
$GL_n(k)$, then $G$ is generated by bi-reflections (elements such
that rank($g-I) \leq 2$). Following our generalization of
reflections, a natural generalization of bi-reflection to a
regular algebra $A$ of dimension $n$ is to call a graded
automorphism $g$ of $A$ a quasi-bi-reflection if its trace has the
form:
\[Tr_A(g,t) = \frac{1}{(1-t)^{n-2} q(t)}\]
where $n$ is the GK-dimension of $A$ and $q(1) \neq 1$.  We have
constructed some examples  that suggest that this is a reasonable
definition (the fixed ring is a commutative complete
intersection).  As in the case of reflections, there are ``mystic
quasi-bi-reflections" (quasi-bi-reflections that are not
bi-reflections of $A_1$).  The notion of bi-reflection may be
useful in determining the proper notion of a non-commutative
complete intersection.

\subsection*{Hopf actions}
One can replace a finite group $G$ acting on an Artin-Schelter
algebra $A$ by a semi-simple Hopf algebra $H$ acting on $A$ \cite{Mon2}
and study properties of $A^H$.  We will report
some results on this case in [KKZ2].

\subsection*{Quotient division algebras}
When $A$ is a Noetherian domain and $G$ is a finite group of
automorphisms of $A$, then $G$ acts on $Q(A)$, the quotient
division ring of $A$.  By \cite[Theorem 5.3]{Mon1} it is known that
$Q(A)^G = Q(A^G)$.  The classical Noether problem is to determine
which linear finite group actions on $k[x_1, \cdots, x_n]$ have
rational fields of invariants, hence it is a natural question to
determine conditions when $Q(A)^G \cong Q(A)$.   Alev and Dumas
have shown that if $G$ is a linear finite abelian group of
automorphisms of $D_n({\mathbb C})$, the quotient division algebra
of the Weyl algebra $A_n({\mathbb C})$, then $D_n({\mathbb C})^G
\cong D_n({\mathbb C})$ \cite{AD1} (and for any finite group with
$n=1$ \cite{AD2}).  One could investigate similar questions for
the quotient division algebras of Artin-Schelter regular algebras.

\begin{example}
\label{xxex7.1}
Let $A$ be the Jordan plane $k_J[x,y]$, the algebra generated by
$x$ and $y$ with relation $xy-yx=x^2$.  We have noted that $A$
is rigid, so that $A^G$ is never isomorphic to $A$ for any
non-trivial finite group of automorphisms. The quotient division
algebra is
$$Q(A)=Q(\mathbb{C}\langle x,y^{-1}\rangle) =
Q(A_1(\mathbb{C}))= D_1(\mathbb{C}).$$
Let $G = \langle g\rangle $ be the group of automorphisms
generated by the automorphism of $A$ given by $g(x) = -x$ and
$g(y) = -y$.  Notice that $g$ induces an automorphism of
$\mathbb{C}\langle x,y^{-1}\rangle = A_1(\mathbb{C})$, so that
by \cite{AD2},
$$Q(A)^G= Q(\mathbb{C}\langle x,y^{-1}\rangle)^G
= Q(A_1(\mathbb{C}))^G \cong D_1(\mathbb{C}).$$
In this case we have $Q(A)^G \cong Q(A)$ even though $A^G$ is
not isomorphic to $A$.

For $A = k_{-1}[x,y]$ and $g$ the automorphism of $A$ given by
$g(x) = -x$ and $g(y) = y$, the invariant subring $A^G$ is the
commutative polynomial ring $k[x^2,y]$. In this case $Q(A)^G =
Q(A^G) = k(x^2,y)$ is not isomorphic to $Q(A)$.  Unlike the
commutative case, our more general notion of reflection groups
means that even when $G$ is a reflection group  $Q(A)^G$ need not
be isomorphic to $Q(A)$.
\end{example}

This paper gives a number of algebras where $A^G$ is never
isomorphic to $A$, so it would be interesting to determine
(a) when $Q(A)^G$ is isomorphic to $Q(A)$, and (b) when $Q(A)^G$
is isomorphic to $Q(B)$ for an Artin-Schelter regular algebra
$B$.

\section*{Acknowledgments}
The authors thank Paul Smith and Ralph Greenberg for several
useful discussions and valuable comments. The authors also thank
the referee for valuable suggestions. J.J. Zhang is supported by
the NSF and the Royalty Research Fund of the University of
Washington.

\providecommand{\bysame}{\leavevmode\hbox
to3em{\hrulefill}\thinspace}
\providecommand{\MR}{\relax\ifhmode\unskip\space\fi MR
}
\providecommand{\MRhref}[2]{%

\href{http://www.ams.org/mathscinet-getitem?mr=#1}{#2}
}
\providecommand{\href}[2]{#2}

\end{document}